\documentclass[11pt]{article}

\usepackage[english]{babel}

\usepackage{amsmath,amsthm, amssymb, bbm}
\usepackage{graphicx}
\usepackage{fullpage}

\newcommand{\R}{{{\mathbb {R}}}}
\newcommand{\C}{{{\mathbb C}}}
\newcommand{\Z}{{\mathbb Z}}

\newcommand{\U}{{\mathbb U}}
\newcommand{\HH}{{\mathbb H}}

\newtheorem{Theorem}{Theorem}

\newtheorem {lemma} [Theorem]    {Lemma}
\newtheorem {corollary}  [Theorem]    {Corollary}

\newtheorem {proposition}[Theorem]    {Proposition}
\newtheorem {theorem}[Theorem]    {Theorem}

\newcommand{\eps}{\epsilon}

\newcommand{\comment}[1]{}

\newcommand{\nucle}{\nu^{\rm cle}}

\let\oldmarginpar\marginpar
\renewcommand\marginpar[1]{\-\oldmarginpar[\raggedleft\scriptsize #1]%
{\raggedright\scriptsize #1}}

\newcommand{\de}{{\mathrm d}}
\newcommand{\ind}{{\mathbbm{1}}}

\newcommand{\PR}{{\mathbf P}}
\newcommand{\EX}{{\mathbf E}}

\begin{document}

\title{The nested simple conformal loop ensembles in the Riemann sphere}
\author{Antti Kemppainen \and Wendelin Werner}
\date {University of Helsinki  and ETH Z\"urich}
\maketitle

\begin {abstract}
Simple conformal loop ensembles (CLE) are a class of random collection of simple non-intersecting loops that are of particular interest in the study of conformally invariant systems. 
Among other things related to these CLEs, we prove the invariance in distribution of their nested ``full-plane'' versions under the inversion $z \mapsto 1/z$.
\end {abstract}

\section{Introduction}

In \cite {Sh,ShW}, a one-dimensional natural class of random collections of simple loops in simply connected domains called Conformal Loop Ensembles has been defined and studied.
We refer to the introduction of \cite {ShW} for a detailed account of the motivations that lead to their study. 
There are two basically equivalent (i.e. defining one enables to define the other one) versions of these simple CLEs, depending on whether one allows loops to be nested (i.e. one loop can surround another loop) or not. Let us recall various definitions and basic features of the latter (i.e. the non-nested) ones:

Such a CLE is defined in a simply connected planar domain $D$ (with $D \not= \C$) and it is a random countable collection $\Gamma = \{ \gamma_j, j \in J\}$ of simple loops
that are all contained in $D$, that are disjoint (no two loops intersect) and non-nested (no loop in this collection surrounds another loop in this collection). Furthermore,
the law of this random collection of loops is invariant under any conformal transformation from
$D$ onto itself, and the image of $\Gamma$ under any given conformal map from $D$ onto some other domain $D'$ is a CLE in $D'$. 
The laws of CLEs can be characterized by an additional condition, called ``Markovian exploration'' that is described and discussed in \cite {ShW}.  

Alternatively, see also \cite {ShW}, 
one can view CLEs as the collections of outer boundaries of outermost clusters in Poissonian collections of Brownian loops in $D$. 
Roughly speaking, one considers a Poissonian collection of Brownian loops in $D$. As opposed to the previous CLE loops, the Brownian loops are not simple, and they are 
allowed to overlap and intersect (and they often do, since they are sampled in a Poissonian -- basically independent -- way). Then one looks at the connected components of the unions of all these loops (i.e. one
hooks up intersecting Brownian loops into clusters). It turns out that when the intensity of this Poisson collection of Brownian loops is not large, then there are several (in fact infinitely many) such clusters. Then, one only keeps the outer boundaries of these clusters (that turn out to be simple loops) and finally only keeps the outermost ones (as some clusters can surround others), one obtains a random collection of non-nested simple loops in $D$.
It turns out that it is a CLE, and that this procedure (letting the intensity of the Brownian loops vary) does in fact construct all possible CLE laws.

A third description description relies on Oded Schramm's SLE processes \cite {Sch}. It turns out that the loops in a CLE are very closely related to SLE$_\kappa$ curves, where the  parameter 
$\kappa$ lies in the interval $(8/3, 4]$ (there is one CLE law for each such $\kappa$, this is called the CLE$_\kappa$), see again \cite {ShW}. This relation will be also useful in the present paper, as it is the one that exhibits some inside-outside symmetry property of the law of the loops. The precise SLE-based construction of the CLEs goes via SLE-based exploration tree (as explained in \cite {Sh}) or via a Poisson point process of SLE bubbles (see \cite {ShW}).

Finally, there is also a close and important relation between CLEs and the Gaussian Free Field (see e.g. \cite {MS1,MS2,MS3,MS4} and the references therein) that we will briefly mention below, but, as opposed to the previous descriptions, we will not build on it in the present paper.

It is noteworthy to stress that these definitions of CLE all a priori take place in simply connected domains with boundary. 

\begin{figure}[ht]
\begin{center}
\includegraphics[scale=0.7]{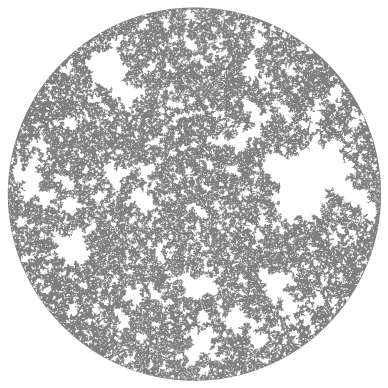}
\end{center}
\caption {A simple non-nested CLE$_4$ in the unit disc (simulation by D.B. Wilson): The loops are the boundaries of the white islands and they are not nested.}
\end {figure}

\medbreak

These loop models are of interest, in particular because they are the conjectural scaling limits of various discrete lattice models. For instance, the loops of the CLE should be the scaling limit of the outermost interfaces in  various models from statistical physics 
(such as the critical Ising model) where some particular boundary conditions are imposed on the boundary of the (lattice-approximation of the) domain $D$. 
Loosely speaking, the boundary of the domain is therefore playing itself the role of an interface i.e. of another loop.
This leads to the very natural definition of the {\em nested CLE} in the domain $D$ which is defined from a simple non-nested CLEs in an iterative i.i.d. fashion (like for a tree-like structure): 
Sample first a non-nested CLE in $D$, then sample independent CLEs in the inside each of this first generation CLE loops and so on. This defines, for each $\kappa$ in $(8/3,4]$ and each domain $D$, a {\em nested CLE$_\kappa$}. This is again a conformally invariant collection of disjoint loops in $D$ as before, but where each given point $z$ in $D$ is now typically surrounded by infinitely many nested loops.
Conversely, if we are given a nested CLE sample in a simply connected domain, one just has to take its outermost loops to get a (non-nested) CLE sample.
These nested CLEs are conjecturally the scaling limits of 
the joint laws of all the interfaces, including all the nested generations,
of a wide class of 
two-dimensional models from statistical physics, such for instance as the $O(n)$ models.


\medbreak

\begin{figure}[ht]
\begin{center}
\includegraphics[scale=0.4]{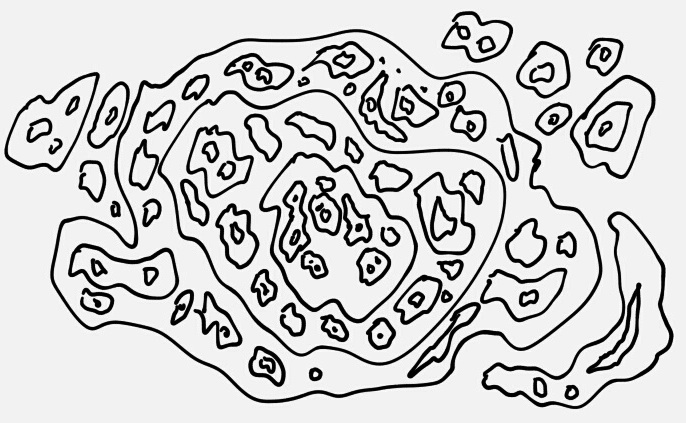}
\end{center}
\caption {Sketch of a nested CLE}
\end {figure}

One of our goals in the present paper is to study some properties of the natural version of these nested CLE 
defined in the entire plane. As we shall see in the first part of the present paper 
(this construction has been independently also written up in \cite {MWW}), the definition of the full-plane generalization of nested CLE is not 
a difficult task (building for instance on the Brownian loop-soup approach to the simple CLEs): One considers the limit when $R \to \infty$ of a 
nested CLE$_\kappa$ defined in the disc of radius $R$ around the origin, and shows that for any fixed $r$, the law of the picture restricted to this disc of radius $r$ converges as $R \to \infty$.
More precisely, one can show that it is possible to couple the nested CLEs in two very large discs of radius $R$ and $R'$ such that with a very large probability $p$, they coincide inside the disc of radius $r$ (i.e. $p$ tends to $1$ when $R, R'$ go to infinity).
Note that by scale-invariance, this procedure is equivalent to looking at the picture of a nested CLE in the unit disc and zooming at the law of the picture in the neighborhood of the origin.
It is then easy to see from this construction that the law of this ``full-plane'' family of nested random loops is translation-invariant and scale-invariant. 

However, with this definition, one property of this full-plane CLE turns out to be not obvious to establish, namely its invariance (in distribution) under the inversion $z \mapsto 1/z$.
Indeed, in the nesting procedure, there is a definite 
inside-outside asymmetry in the definition of CLEs. One always starts from the boundary of a simply connected domain, and discovers the loops ``from their outside'' (i.e. the point at infinity in the Riemann sphere plays a very special role in the construction). 

On the other hand, invariance of the full-plane CLEs under inversion is a property that is expected to hold. Indeed:
\begin {itemize}
 \item The discrete $O(n)$ models that are conjectured to prove to these CLEs have a full-plane version, for which one expects such an inside-outside symmetry. In the particular case of the Ising model (which is the $O(1)$ model) that is known to be conformally invariant in the scaling limit (see \cite {CS,CD}) and should therefore correspond to CLE$_3$, 
 there is a full-plane version of the discrete critical Ising model that should in principle be invariant under $z \mapsto 1/z$ in the scaling limit as well.
 \item In the case where $\kappa=4$, the nested CLE$_4$ can be viewed as level (or jump) lines of the Gaussian Free Field, and it is possible (though we will not do this in the present paper) to define the full-plane CLE$_4$ in terms of a full-plane version of the Gaussian Free Field (which is then defined up to an additive constant, so a little care is needed to justify this -- in particular, additional randomness is needed in order to define the nested CLE$_4$ from this full-plane GFF), and to see that the obtained CLE$_4$ is indeed invariant under $z \mapsto 1/z$, using the strong connection between CLE$_4$ and the GFF (in particular, the fact that CLE$_4$ is a deterministic function of the GFF when defined in a simply connected domain) derived in \cite {SchSh,Dub}.
\end {itemize}
While the previous provable direct connections of the full-plane CLE$_3$ and CLE$_4$ to the Ising model and the Gaussian Free Field respectively indicate quite direct roadmaps towards establishing their invariance under inversion (the CLE$_4$ case is actually quite easy), it is not immediate to adapt those ideas to the case of the other CLE$_\kappa$'s for $\kappa \in (8/3,4]$ (note for instance that the coupling between other CLEs and the GFF \cite {MS1,MS3} involves additional  randomness that does not seem to behave so nicely with respect to inversion).

One of our two main goals in this paper is to establish the following result: 

\begin {theorem}
\label {mainthm}
 For any $\kappa \in (8/3, 4]$, the law of the nested CLE$_\kappa$ in the full plane (as described above)  is invariant under $z \mapsto 1/z$ (and therefore under any 
 M\"obius transformation of the Riemann sphere).
\end {theorem}

Since the law of nested CLE on the Riemann sphere is fully M\"obius invariant
 and hence the law doesn't depend on the choice of the root point,
 it makes sense to call it the \emph{Conformal Loop Ensemble of the Riemann sphere} with
 parameter $8/3 < \kappa \leq 4$ and denote it by CLE$_\kappa(\hat{\C})$. 
 One way to characterize this family of CLE's is that they
 are random collection of loops such that the loops are nested, pair-wise disjoint and simple
 and that they have the following \emph{restriction property}:
 if $A \subset \hat{\C}$ is a closed subset of the Riemann sphere with simply connected
 complement and if $z_0 \in \hat{\C} \setminus A$, then define the set $\tilde{A}$ to be the union of $A$
 and all the loops that intersect $A$ together with their interiors --- as seen from $z_0$, i.e.
 $z_0$ lies outside of these loops. Then the property, which we call restriction property
 of CLE$_\kappa(\hat{\C})$, is that the restriction of the CLE$_\kappa(\hat{\C})$ 
 to the loops that stay in $U = \C \setminus \tilde{A}$ is the nested CLE$_\kappa$
 in $U$.

One motivation for the present work comes from the fact that, as indicated for instance by the papers of Benjamin Doyon \cite {Do}, it is possible to use such nested CLEs in order to 
provide explicit probabilistic constructions and interpretations of various basic concepts in Conformal Field Theory (such as the bulk stress-energy tensor). The paper \cite {Do} for instance builds on 
some assumptions/axioms about nested CLEs, that we prove in the present paper.

\medbreak

An instrumental idea in the present paper will be to use a ``full-plane'' version of a variant of the Brownian loop soup, where one only keeps the outer boundary of each Brownian loop instead of the whole Brownian loop. It turns out (this fact had been established in \cite {W}) that this soup of overlapping simple loops is invariant under $z \mapsto 1/z$, and that (as opposed to the Brownian loop soup itself) it creates more than one cluster of loops when the intensity of the soup is subcritical. We will refer to this loup soup as the SLE$_{8/3}$ loop soup.
This is a random full-plane structure that is indirectly related to CLE, even though it is not the nested CLE itself. 

Actually, \emph{the other main purpose of the present paper} 
will be to derive and highlight properties of this particular full-plane structure that we think is interesting on its own right. We shall for instance see that outer boundaries of such clusters and inner boundaries are described by exactly the same intensity measure. More precisely, if one considers a full-plane SLE$_{8/3}$ loop soup, one can construct its clusters, and define those clusters $K_j$ of loops that surround the origin. Each one has an outside boundary $\gamma_j^e$ and an inside boundary $\gamma_j^i$. 
Then, one can define the intensity measures $\nu^i$ and $\nu^e$ by 
setting for each measurable set $A$ of simple loops (As the sigma algebra, we use always use
the usual sigma algebra of events of staying in annular regions, see Section~3 of \cite{W}.)
$$ \nu^i ( A) = \EX ( \# \{ j \, : \, \gamma_j^i  \in A\} ) \hbox { and }Ê \nu^e ( A) = \EX ( \# \{ j \, : \, \gamma_j^e \in A\} ) .$$
Similarly, for a full-plane CLE (with the $\kappa$ corresponding to the intensity of the loop soup), one can define the intensity measure $\nucle$ of the loops that surround the origin. Then, 
\begin {theorem}
\label {secondthm}
For some constant $\alpha=\alpha(\kappa)$, one has $\nu^i = \nu^e = \alpha \times  \nucle$.
\end {theorem}
In fact, the proof will go as follows (even if we will not present the arguments in that order): One first directly proves that $\nu^i = \nu^e$ (which will be the core of our proofs), and then deduce Theorem \ref {secondthm} from it using the inversion invariance of the SLE$_{8/3}$ loop soup, and then finally deduce Theorem \ref {mainthm} from Theorem \ref{secondthm}. 

This paper will be structured as follows. First, we will recall the basic properties of the SLE$_{8/3}$ loop soup, and deduce from it the definition and some first properties of the full-plane CLE$_\kappa$'s. Then, we will build on some aspects of the exploration procedure described in \cite {ShW} to define CLEs using SLE$_\kappa$ loops, and use some sample properties of SLE paths in order to derive Theorem \ref {secondthm}. 

\section{Chains of loops and clusters from the SLE$_{8/3}$ loop soup}

\subsection{Loop soups of Brownian loops and of SLE$_{8/3}$ loops}
The Brownian loop soup in $\C$ with intensity $c$ is a Poisson point process in the plane with intensity $c \mu$, where $\mu$ is the Brownian loop-measure defined in \cite {LW}. A sample of the Brownian loop soup in $D$ can be obtained from a sample of the Brownian loop soup in the entire plane, by just keeping those loops that fully stay in $D$. More precisely, if $\beta= \{\beta_j , j \in J\}$ is a Brownian loop soup in the plane with intensity $c$ and if $J_D = \{ j \in J \, : \, \beta_j \subset D \}$, then 
$\beta^D= \{\beta_j, j \in J_D\}$ is a sample of the Brownian loop soup with intensity $c$. 
It is shown in \cite {ShW} that when $c \le 1$, the Brownian loop-soup clusters in $D$ are disjoint, and that their outermost boundaries form a sample of a CLE$_\kappa$ (where $\kappa$ depends on $c$). 
{\bf For the rest of the present paper, the value of $c \in (0,1]$ (and the corresponding $\kappa (c) \in (8/3, 4]$) will remain fixed, and we will omit  them (we will just write CLE instead of CLE$_\kappa$ and loop soup instead of loop soup with intensity $c$).}

If one considers the full-plane Brownian loop soup, then because of the too many large Brownian loops (and the fact that infinitely many large Brownian loops in the loop soup do almost surely intersect the unit circle), it is easy to see that there exists almost surely only one dense cluster of loops. The Brownian loop soup does therefore not seem so well-adapted to define a full-plane
structure. 

The following observations will however be useful: Firstly, when $D$ is simply connected, define for each Brownian loop $\beta_j$ for $j \in J_D$, its outer boudary $\eta_j$ (the outer boundary of a Brownian loop is almost surely a simple loop, see \cite {W} and the references therein). Then, consider the outer boundaries of outermost clusters of loops defined by the family of simple loops $\eta_D= \{\eta_j, j \in J_D\}$ (instead of $\beta_D$). Clearly, this defines the very same collection of non-nested simple loops as the outer boundaries of outermost clusters of $\beta_D$, and it is therefore a CLE. 
Secondly, it is shown in \cite {W} that the family $\eta = \{\eta_j, j \in J\}$ is a Poisson point process of SLE$_{8/3}$ loops, and that this random family is invariant (in law) under any M\"obius transformation of the Riemann sphere (in particular under $z \mapsto 1/z$). This yields a non-trivial ``inside-outside'' symmetry of Brownian loop boundaries (the proof in \cite {W} is based on the fact that this outer boundary can be described in SLE$_{8/3}$ terms). 
Hence, we see that the CLE$_\kappa$ is also the collection of outer boundaries of outermost cluster of loops of an SLE$_{8/3}$ loop soup in $D$ (that can itself be viewed as the restriction of a full-plane SLE$_{8/3}$ loop soup to those loops that stay in $D$).

Finally, we note that the outer boundary $\eta_j$ of the Brownian loop $\beta_j$ is clearly much ``sparser'' than $\beta_j$ itself (the inside is empty...). As indicated in \cite {W}, it turns out that if one considers the soup $\eta$ of SLE$_{8/3}$ loops in the entire plane, then (for $c \le 1$) the clusters will almost surely all be finite and disjoint. Here is a brief justification of this fact: 

\begin {itemize}
 \item Note first that if we restrict ourselves to a (subcritical i.e., $c \le 1$) loop soup $\eta_\U$ in the unit disc $\U$, then the outer boundaries of outermost loop-soup clusters do form a CLE$_\kappa$. Therefore, the outermost cluster-boundary $\gamma$ (in the CLE in $\U$) that surrounds the origin is almost surely at positive distance of the unit circle. Hence, for some positive $\eps$, 
$$ \PR ( d ( \gamma , \partial \U )  > \eps ) \ge 1/2 .$$
Let us call $A_1$ this event $\{d ( \gamma , \partial \U ) > \eps\}$.
\item The total mass (for the SLE$_{8/3}$ loop measure defined in \cite {W}) of the set of loops that intersect both $\partial U$ and $(1-\eps) \partial \U$ is finite. This can be derived in various ways. One simple justification uses the description of this measure as outer boundaries of (scaling limits) of percolation clusters (see \cite {W}), and the fact that the expected number of critical percolation clusters that intersect both $R \partial \U$ and $(1-\eps)R \partial \U$ is finite and bounded independently of $R$ (this is just the Russo-Seymour-Welsh estimate) and therefore also in the $R \to \infty$ limit. Alternatively, one can do a simple SLE$_{8/3}$ computation. 
Hence, if we perform a full-plane SLE$_{8/3}$ loop soup, then with positive probability, no loop in the  soup with intersect both 
$\partial U$ and $(1-\eps) \partial U$. Let us call $A_2$ this event. 
\item The events $A_1$ and $A_2$ are independent ($A_1$ is measurable with respect to the set of loops in the loop soup that stay in $\U$ and $A_2$ is measurable with respect to the set of loops in the loop soup that intersect $\partial \U$). Hence, the probability that $A_1$ and $A_2$  hold simultaneously is strictly positive. This implies that with positive probability, there exists a cluster of loops in the full-plane SLE$_{8/3}$ loop soup, that surrounds the origin and is contained entirely in the unit disc.
\item It follows immediately (via a simple $0-1$ law argument, because the event that there exists an unbounded loop-soup cluster does not depend on the set of loops that are contained in $R \U$ for any $R$, and is therefore also independent of the loop soup itself) that almost surely, all clusters in this soup are bounded (if not, the distance between the origin and the closest infinite cluster is scale-invariant and positive).
\item The fact that the clusters are almost surely all disjoint can be derived in a rather similar way (just notice that if two different full-plane loop-clusters had a positive probability to be at zero distance from each other, then the same would be true for two CLE loops in the unit disc, with positive probability, and we know that this is almost surely not the case.
\end {itemize}

To sum things up: For any given $c \le 1$, the full-plane SLE$_{8/3}$ loop soup defines a random collection of clusters $(K_i, i \in I)$ that is {\bf invariant in distribution under any 
M\"obius transformation of the Riemann sphere} (including the inversion $z \mapsto 1/z$), and the boundaries (inner and outer boundaries) of these clusters 
are closely related to SLE$_\kappa$ paths for $\kappa = \kappa (c) \in (8/3, 4]$.

\subsection{Markov chains of nested clusters and of nested loops}

We are now going to pick two points on the Riemann sphere, namely the origin and infinity (but by conformal invariance, this choice is not restrictive), and we are going to focus only on those clusters 
that disconnect one from the other i.e. that surround the origin. Recall that almost surely, both the origin and infinity are not part of a cluster (and scale-invariance shows that there exist almost surely a countable family of
loop-soup clusters that disconnect $0$ from infinity). 
We can order those clusters that disconnects infinity from the origin ``from outside to inside''. We will denote this collection by $(K_j, j \in J)$ 
where $J \subset I$ is now a decreasing bijective image of $\Z$ (each $j$ in $J$ has therefore a successor denoted by $j+1$).

\begin{figure}[ht]
\begin{center}
\includegraphics[scale=0.5]{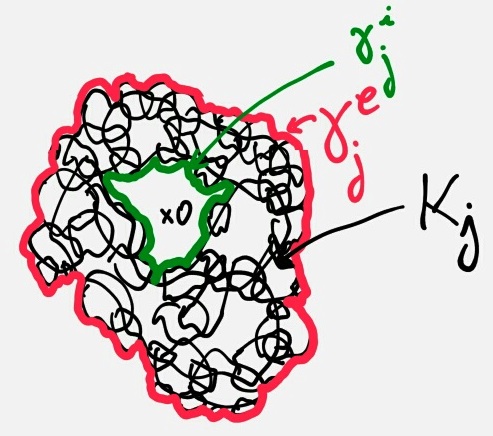}
\end{center}
\caption {A SLE$_{8/3}$ loop cluster, with its outer and inner boundary}
\end {figure}

The boundaries of the complement of each $K_j$ consists of countably many loops, two of which (corresponding to the connected components $O_j^i$ and $O_j^e$ of $\C \setminus K_j$ that respectively contain the origin and infinity) surround the origin. We will call these boundaries $\gamma_j^i$ and $\gamma_j^e$. One therefore has a nested discrete sequence of loops, when $j$ in $J$, then 
$$\gamma_{j}^e \succ \gamma_j^i \succ \gamma_{j+1}^e$$
where $\gamma \succ \gamma'$ means that $\gamma$ surrounds $\gamma'$ (we however allow here the possibility that $\gamma$ intersects $\gamma'$ -- indeed, for small $c$, it happens that 
for a positive fraction of the $j$'s, the inner and outer boundaries $\gamma_j^i$ and $\gamma_j^e$ of $K_j$ do intersect).  

The scale invariance of the loop soup, as well as the fact that the expected number of clusters that surround the origin and have diameter between $1$ and $2$ say is finite, shows immediately that we can
define three infinite measures $\nu$, $\nu^i$ and $\nu^e$ that correspond to the intensity measure of the families $(K_j)$, $(\gamma_j^i)$ and $(\gamma_j^e)$ respectively. 
In other words, for any measurable family $L$ of loops (see e.g. \cite {W} for details on the $\sigma$-field that one can use), 
$$ \nu^i (L) = \EX \bigg( \sum_{j \in J} \ind_{\gamma_j^i \in L} \bigg)$$
(and the analogous definition for the measure $\nu^e$ on outer loops and for the measure $\nu$ on clusters of loops, defined on an appropriately chosen $\sigma$-field).
Clearly, these measures are scale-invariant i.e. for any set $L$ of loops and any positive $\lambda$, 
$$\nu^i  (L) = \nu^i ( \{ \gamma \, : \, \lambda \gamma \in L \}).$$
Additionally, these measures have the following inversion relations which play an important role
later.

\begin{proposition}\label{prop: measures under z mapsto 1/z}
The measure $\nu$ is invariant under $z \mapsto 1/z$
and the image of the measure $\nu^i$ under $z \mapsto 1/z$ is $\nu^e$.
\end{proposition}

\begin{proof}
The claim follows from the fact that the full-plane SLE$_{8/3}$ 
loop soup is invariant under inversion.
\end{proof}


\medbreak

Let us define three Markov kernels that are heuristically correspond to 
the mapping $K_{j} \mapsto K_{j+1}$, $\gamma^i_{j} \mapsto \gamma^i_{j+1}$ 
and $\gamma_{j}^i \mapsto \gamma_{j+1}^e$.
Note that in the definition of these chains, we always explore from outside 
to inside and from one cluster to the next one.

More rigorously, for any simply connected domain $D \not= \C$ that contains the origin, sample an SLE$_{8/3}$ loop soup in $D$ and denote by ${\cal L}^{K}_D$ (respectively ${\cal L}^{i}_D$ and ${\cal L}^e_D$) the law of the outermost cluster that surrounds the boundary (resp. the inner boundary of this outermost cluster and the outer boundary of the outermost cluster).
When $A$ is a compact set that surrounds the origin, denote (whenever it exists) by $D(A)$ the connected component of the complement of $K$ that contains the origin. 
Then, the kernels are defined as follows: Consider  $Q^{\to K} (A,  \cdot) := {\cal L}^K_{D(A)} ( \cdot)$ and similarly 
$Q^{\to i} (A, \cdot) := {\cal L}^i_{D(A)} (\cdot)$ and 
 $Q^{\to e} (A, \cdot) := {\cal L}^e_{D(A)}(\cdot)$.

\begin{figure}[ht]
\begin{center}
\includegraphics[scale=0.4]{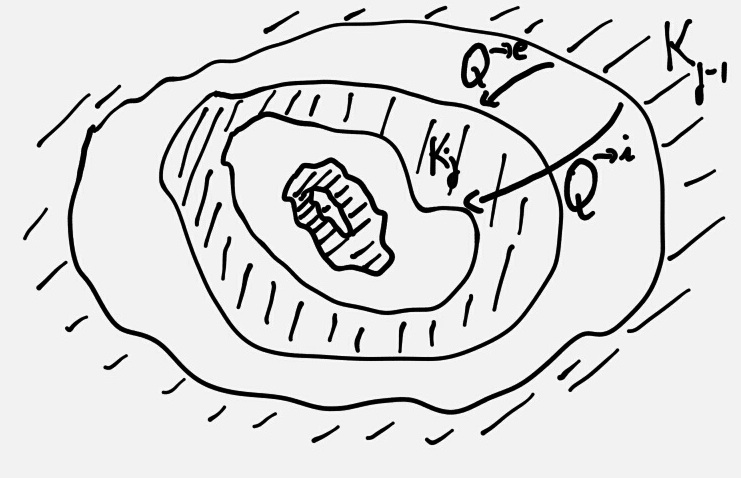}
\end{center}
\end {figure}

Let us now consider a full-plane SLE$_{8/3}$ loop soup, and take two different simply connected domains $D$ and $D'$ that contain the origin. For $D$ (respectively $D'$), we restrict
the full-plane loop soup to the set of loops that are contained in $D$ (resp. $D'$). Hence, we have now three families of nested clusters that surround the origin:
\begin {itemize}
 \item The clusters $(K_j, j \in J)$ of the full-plane loop soup.
 \item The clusters $(K_n^D, n \ge 1)$ and $(K_{n'}^{D'}, n' \ge 1)$ corresponding to the loop soups in $D$ and $D'$ respectively. These two sequences can ve viewed as Markov chains with kernel $Q^{ \to K}$ started from $\partial D$ and $\partial D'$ respectively. Similarly, their inner boundaries are Markov chains with kernel $Q^{\to i}$.
\end {itemize}
The properties of the full-plane SLE$_{8/3}$ loop soup imply immediately that almost surely, there exists $j_0$, $n_0$ and $n_0'$ so that for all $n \ge 0$, 
\begin {equation}
K_{j_0 + n } = K_{n_0 +n }^D = K_{n_0' + n }^{D'}.
\label {coupl}
\end {equation}
Indeed, almost surely, for small enough $\epsilon$, no loop in the loop soup does intersect both the circle of radius $\epsilon$ around the origin and $\C \setminus D$ or $\C \setminus D'$, which implies that the ``very small'' clusters that surround the origin are the same in all three pictures. 


\subsection{Shapes of clusters and of loops}

Note that one can  decompose the information provided by a loop $\gamma$ (or a set $K$) that surrounds the origin into two parts (we now detail this in the case of the loop): 
\begin {itemize}
 \item Its ``size'', for instance via the log-conformal radius $\rho (\gamma)$ of its interior, seen from the origin (such that the Riemann mapping $\Phi_\gamma$ from the unit disc onto the interior of $\gamma$ such that $\Phi_\gamma(0) = 0$ and $\Phi_\gamma' (0) \in (0, \infty)$ satisfies $\Phi_\gamma' (0) = \exp ( \rho (\gamma))$.
 \item Its ``shape'' $S(\gamma)$ i.e. its equivalence class under the equivalence relation $$\gamma \sim \gamma'
 \Longleftrightarrow \hbox {There exists some positive $\lambda$ for which $\gamma = \lambda \gamma'$}.$$
\end {itemize}
For a shape $S$ and a value $\rho$, we define $\gamma ( \rho , S)$ to be the only loop with shape $S$ and log-conformal radius $\rho$.
The scale-invariance of $\nu^i$ implies immediately that there exists a constant  $a_i$ and a probability measure $P^i$ on the set of shapes so that $\nu^i$ is the image of the product measure 
$a_i d\rho \otimes P^i$ under the mapping $(\rho, S) \mapsto \gamma (\rho, S)$.

The same of course holds for $\nu^e$, which defines a constant $a_e$ and a probability measure $P^e$, and for $\nu$ that defines a constant $a_K$ and a probability measure $P^K$.
We can note that  for any $R$, for a full-plane SLE$_{8/3}$ loop-soup sample, the number of inside cluster boundaries and the number of exterior cluster boundaries that are included in the annulus between the circles of radius $1$ and $R$ can differ only by at most $1$ from the number of interior cluster boundaries in this annulus 
(because the loops $\gamma_j^i$ and $\gamma_j^e$ are alternatively nested). It follows (letting $R \to \infty$ and looking at the expected number of such respective loops) that $a_e=a_i=a_K$ (and we will denote this constant by $a$).

\medbreak

We can note that the three kernels $Q^{\to i}$, $Q^{\to K}$, $Q^{\to e}$ induce  kernels $\tilde Q^{ \to i}$, $\tilde Q^{\to K}$ and $\tilde Q^{ \to e}$ on the set of shapes 
(because the former kernels are ``scale-invariant'').
The coupling property (\ref {coupl}) implies immediately that $P^K$ and $P^i$ are the unique stationary distributions for $\tilde Q^{\to i}$ and $\tilde Q^{\to K}$. 
It follows that $\nu^i$ and $\nu$ are (up to a multiplicative constant) the only scale-invariant measures that are invariant under $Q^{\to i}$ and $Q^{\to K}$ respectively.

\section{Full-plane CLE and reversibility}

In the next two subsections, we will describe the construction of the full-plane CLEs. This has been independently and in parallel written up also in \cite {MWW}, where it is used for another purpose.

\subsection{Markov chain of nested CLE loops and its properties}

In order to construct and study the nested CLEs, we will focus on the kernel $Q^{\to e }$ instead of $Q^{\to i}$. 
Let us first collect some preliminary simple facts:

\begin {enumerate}
 \item 
Let us consider first a loop soup in the unit disc. We know a priori that the log-conformal radius of $\gamma^e_{1}$ is not likely to be very small: For instance, 
 for any positive $x_0$, there exists $c >0$ so that for all $x$,
 \begin {equation}
  \label {lem1}
 \PR \left( \rho (\gamma^e_1) \le -x-x_0 \right) \le e^{-cx_0} \, \PR \left( \rho ( \gamma_1^e) \le -x \right).
 \end {equation}
 Indeed, if $\rho (\gamma^e_1) \le -x-x_0$, then on the one hand, the annulus $\{ z \, : \, e^{-x_0} < |z| < 1 \}$ does not contain an SLE$_{8/3}$ loop in the loop soup (which is an event of probability strictly smaller than one), and on the other 
 hand, if we restrict the loop soup to the disc $e^{-x_0} \U$, the outermost loop-soup cluster boundary $\tilde \gamma_1^e$ that surrounds the origin satisfies 
 $ \rho ( \tilde \gamma_1^e ) \le \rho (\gamma_1^e) \le -x-x_0$. But these two events are independent, and the laws of 
  $ \rho ( \tilde \gamma_1^e ) + x_0$ and of  $\rho (\gamma_1^e)$ are identical by scale-invariance, so that (\ref {lem1}) follows.
\item
Consider a sequence $(\xi_n, n \ge 1)$ of i.i.d. positive random variables such that for some $x_0$ and $c >0$, and for all $x$,
$\PR ( \xi_1 > x + x_0 ) \le e^{-cx_0} \, \PR ( \xi_1 > x)$. Define 
$S_n = \xi_1+ \ldots + \xi _n$ and for all $y >0$, the overshoot at level $y$ i.e. $O(y) = \min \{ S_n - y \, : \, n \ge 1 \hbox { and } S_n > y \}$. 
Then, for all $M$ that is a multiple of $x_0$,
\begin {equation} 
\label {lem2}
\PR ( O(y) \ge M ) \le e^{-cM}.
\end {equation}
Indeed, if we suppose that $M$ is a multiple of $x_0$, then
\begin {align*}
 \PR ( O( y) \ge M ) 
 & =  \sum_{n \ge 0} \PR ( S_n < y \hbox { and } \xi_{n+1} \ge M +y  - S_n ) \\ 
 & =  \sum_{n \ge 0} \EX \left[ \,
        \ind_{S_n < y} \, \PR ( \xi_{n+1} \ge M +y  - S_n \,|\, \sigma(\xi_1, \ldots, \xi_n)  ) \,\right] \\ 
 & \leq  \sum_{n \ge 0}  e^{-cM} \, \EX \left[ \,
        \ind_{S_n < y} \, \PR ( \xi_{n+1} \ge y  - S_n \,|\, \sigma(\xi_1, \ldots, \xi_n)  ) \,\right] \\
 & =  e^{-cM} \sum_{n \ge 0} \PR ( S_n < y \le S_{n+1})  = e^{-cM}
 \end {align*}
where $\sigma(\xi_1, \ldots, \xi_n) $ is the sigma algebra generated by the random variables
$\xi_k$, $k = 1,2,\ldots, n$.
Note that (\ref {lem2}) shows that for a large enough given $M$, $P( O(y) \le M ) \ge 1/2$ (independently of $y$).
\item
Let us briefly recall how to define a nested CLE in the simply connected domain $D$ (with $D \not= \C$). We first sample a simple CLE, that defines a countable collection of disjoint and non-nested loops in $D$. 
For each $z \in D$, it is almost surely surrounded by a loop denoted by $\gamma_1(z)$ in this CLE (note of course that while for each given $z$, $ \gamma_1 (z)$ almost surely exists, there  exists a random fractal set with zero Lebesgue measure of points  that are surrounded by no loop). In particular, if the origin is in the domain $D$, 
then the loop $\gamma_1 (0)$ is distributed 
like the loop $\gamma_1^e$.

Then, once this first-layer CLE is defined, we repeat (conditionally on this first generation of loops) the same experiment independently inside each of these countable many loops. For each given $z$, this defines almost surely 
a second-layer loop $\gamma_2 (z)$ that surrounds $z$. We then repeat this procedure indefinitely. Hence, for any fixed $z$, we get almost surely 
a sequence of nested loops $(\gamma_n (z), n \ge 1)$. 

Let us now suppose that $0 \in D$. Clearly, if we focus only the loops that surround the origin $(\gamma_1, \gamma_2, \ldots ) := (\gamma_1 (0), \gamma_2 (0), \ldots)$, we get a Markov chain with kernel $Q^{\to e}$. We can now define the random variables 
$\xi_1  = \rho (D) - \rho (\gamma_1)$ and for all $j \ge 2$, $\xi_j = \rho (\gamma_{j-1}) - \rho(\gamma_j)$
corresponding to the successive jumps of the log-conformal radii. These are i.i.d. positive random variables, and combining the previous two items, we see that there exists $M$ and $c$ such that, for all $v < \rho (D)$, if $j_0$ is the first $j$ for which $\rho ( \gamma_j) < v$, then 
\begin {equation}
 \PR ( \rho (\gamma_{j_0}) \le v - M ) \le e^{-cM}.
 \label {lem3prime}
 \end {equation}
In particular, for some given $M$, 
\begin {equation}
\PR ( \rho (\gamma_{j_0}) \le v - M ) \le 1/2.
\label {lem3}
\end {equation}

\item
Let us now consider two bounded simply connected domains $D$ and $D'$ that surround the origin, and try to couple the (non-nested) CLEs in these two domains in such a way that 
the {\em first} loops $\gamma_1$ and $\gamma_1'$ that surround the origin coincide. We will assume in this paragraph that the log-conformal radii of $D$ and $D'$ are not too different i.e. that $$ | \rho (D) - \rho (D') | \le M$$
(where $M$ is chosen as in (\ref {lem3})).

We consider a realization of the SLE$_{8/3}$ loop soup in the full-plane, and then restrict them to $D$ and $D'$ respectively. 
This defines a coupling of the two loops $\gamma_1^e$ and ${\gamma_1^e}'$. Then, for this coupling, 
there exists a positive constant $u$ that does not depend on $D$ and $D'$ so that
\begin {equation}
\PR ( \gamma_1^e = {\gamma_1^e}' ) > u 
\label {lem4}
\end {equation}
Let us now briefly indicate how to 
prove this fact: Clearly, we can assume that $\rho (D') \ge \rho (D)$ (otherwise, just swap the role of $D$ and $D'$), and  
because of scale-invariance, we can assume, without loss of generality that $\rho (D) = 0$ (i.e. that there exists a conformal map $\Phi$ from $D$ onto $\U$ such that $\Phi (0)=0$ and $\Phi'(0)=e^0 = 1$). By Koebe's $1/4$ Theorem, this implies that $\U \setminus 4 \subset D$ and similarly (because $\rho (D') \ge 0$), that $\U / 4 \subset D'$. 

Let us consider a full-plane loop soup of SLE$_{8/3}$ loops.  
Let us first restrict this loop soup to the disc $\U /8$, and define the event that there exists an outer-boundary of a cluster 
 in this loop soup such that its outer boundary  does fully surround $\U/16$ and such that the cluster itself does intersect $\U /16$. 
If such a cluster exists, then it is clearly unique -- we denote it by $K$. Note that at this point,  
we have not required that no other cluster of the loop soup in $\U/8$ surrounds $K$.

Considerations from \cite {ShW} show that such a $K$ indeed exists with a positive probability $u_0$. Furthermore, we can discover this event ``from the inside'' by exploring all loop-clusters of the loop soup that do intersect the disc $\U / 16$. Hence, for any simply connected domain $V$, the event that $K\subset V$ is independent 
from the loops in the full-plane loop soup that do not intersect $V$. It therefore follows easily that (on the event where $K$ exists) the conditional law of the loop soup outside of $K$ (given $K$) is just a SLE$_{8/3}$ loop soup restricted to the outer complement of $K$.

On the other hand, we also know (for instance from \cite {ShW}) that with positive probability (that is bounded from
below independently of the shape of $D$), the outermost cluster $K_1$ in the CLE in $D$ is a subset of $\U /16$. It therefore follows that, conditionally on $K$, if we then sample the loops in $D$ that lie outside of $K$, with a conditional probability that is bounded uniformly away from $0$ (i.e. uniformly larger than some $u_1$), we do not create another cluster of loops that surrounds $K$. Hence, the conditional probability that $K_1 = K$ is greater than $u_0u_1$. 

The same holds for $K_1'$ (using this time the fact that $0 \le \rho (D') \le M_0$), and (when one first conditions on $K$), the events $K_1 = K$ and $K_1' = K$ are positively correlated (they are both decreasing events of the loop soup outside of $K$). 
Hence, we conclude that the conditional probability that $K = K_1 = K_1'$ is bounded away from $0$ uniformly, from which (\ref {lem4}) follows. 

\end {enumerate}

With these results in hand, we can now construct a coupling between nested 
CLEs between any two given  simply connected domains $D$ and $D'$ that surround the origin, in 
such a way that they coincide in the neighborhood of the origin: 

We will first only focus on the two sequences of loops that surround the origin $(\gamma_1, \gamma_2, \ldots)$ and $(\gamma_1', \gamma_2', \ldots)$ that we will construct from the outside to the inside in a ``Markovian way'', and we will couple them in such a way that for some $n_0$ and $n_0'$, $\gamma_{n_0} = \gamma_{n_0'}'$. Then, we will choose 
the two nested CLEs in such a way that they coincide within this loop $\gamma_{n_0}$.

Suppose for instance that $\rho (D) \ge \rho (D')$ (the other case is treated symmetrically) -- note however that we do not assume here that $\rho (D)$ and $\rho (D')$ are close. 
So, our first step is to try to discover some loops $\gamma_m$ and $\gamma_{m'}'$ in the two nested CLEs that have a rather close log-conformal radius. 

We therefore first construct $\gamma_2, \gamma_3, \ldots$ (using the Markov chain $Q^{\to e}$) until 
$\gamma_{m_1}$, where 
$$ m_1 = \min \{ m \ge 1  \, : \, \rho (\gamma_m) < \rho (D')\} .$$ Two cases arise:
\begin {itemize}
 \item Case~1: $\rho (\gamma_{m_1}) \ge  \rho (D') - M$. 
By (\ref {lem3}), we know that this happens with probability at least $1/2$. In this case, these two sets have close enough conformal radii (so that we will then be able to couple  $\gamma_{m_1 + 1}$ with $\gamma_1'$ so  that they coincide with probability at least $u$) and we stop.
 \item Case~2: $\rho (\gamma_{m_1}) < \rho (D')  - M$. Then, we start constructing the loops $\gamma_1', \ldots$ until we find a loop $\gamma_{m_1'}'$
 such that $\rho (\gamma_{m_1'}') < \rho (\gamma_{m_1})$. Again, either the difference between the conformal radii is in fact smaller than $M$, and 
 we stop. Otherwise, we start exploring the loops $\gamma_{m_1+1}, \ldots$ until we find 
 $\gamma_{m_2}$ with $\rho (\gamma_{m_2}) < \rho (\gamma_{m_1'}')$, and so on. At each step, the probability that we stop is at least $1/2$, so that this procedure necessarily ends after a finite number of iterations. 
\end {itemize}
In this way, we almost surely find $\gamma_m$ and $\gamma_{m'}'$ so that $|\rho (\gamma_m) - \rho (\gamma_{m'}')| \le M$. Furthermore, we have not yet explored/constructed the loops inside these two loops. Hence, we can now use (\ref {lem4}) to couple $\gamma_{m+1}$ with $\gamma_{m'+1}'$ so that they are equal with probability at least $u$.

On the part of the probability space where the coupling did not succeed, we start the whole procedure again by continuing to construct loops inwards from these two loops 
$\gamma_{m+1}$ and $\gamma_{m'+1}'$. Again, since this coupling succeeds at each iteration with a probability at least $u$, we finally conclude that almost surely, using this construction, we will eventually find $\bar m$ and $\bar m'$ so that $\gamma_{\bar m} = \gamma_{{\bar m}'}'$.

A final observation is that (because of (\ref {lem3prime})), for this construction  
$$\PR \left( \rho (\gamma_{\bar m}) \,<\, \min (\rho (D), \rho (D')) - x \right) \to 0 $$
as $x \to \infty$, uniformly with respect to all choices of $D$ and $D'$. 

Hence, we have obtained the following result.

\begin {proposition}\label{prop: coupling of nested CLEs}
For any $D$ and $D'$, it is possible to couple the nested CLEs in $D$ and in $D'$ in such a way that almost surely:
\begin {itemize}
 \item There almost surely exists an $n_0$ such that the two nested CLEs 
coincide inside the loop $\gamma_{n_0}$.
\item Furthermore, $ \PR ( n_0 \ge j )$ tends to zero as $j$ tends to infinity,  uniformly over all possible sets $D$ and $D'$ with $\rho (D') \ge \rho (D)$.
\end {itemize}
\end {proposition}

\subsection{Definition of the full-plane CLE}

Proposition~\ref{prop: coupling of nested CLEs} enables us to define and state a few properties of the full-plane CLE.

\begin {itemize}
 \item The law of the part of the nested loop soup in the disc $n \U$ that is contained in a finite ball of radius $r>0$ does converge when $n \to \infty$ to a limit.
 \item For any sequence of domains $D_n$ with $\rho (D_n) \to \infty$, the law of the part of the nested loop soup in the disc $n \U$ that is contained in a finite ball of radius $r>0$ does converge when $n \to \infty$ to the same limit.
\end{itemize}

The first statement is just obtained by noting that the previous proposition shows that it is possible to couple the nested CLE in $n \U$ with the nested CLE in $n'\U$, so 
that with a large probability (that tends to $1$ when $n, n' \to \infty$) they coincide inside the disc of radius $r$. 
The second just follows from the coupling between the CLEs in $D_n$ and $n \U$.

The above convergence enables to define the \emph {full-plane nested CLE} (started from $\infty$) to be the law on nested loops 
in the entire plane that coincide with the limit inside each disc of radius $r$. 
Let $b \in \C$.
We define the \emph{full-plane nested CLE chain to $b$} (from $\infty$) to 
the restriction of  the full-plane nested CLE to those loops that surround $b$.
We use the notation CLE$(a \to \C)$ and CLE$(a \to b)$ for the full-plane nested CLE started from $a$ 
and the corresponding chain to $b$, which
are defined for other $a$ than $\infty$ by a M\"obius transformation.

For this last definition, we used the fact that 
the full-plane CLE is scale-invariant and translation invariant in distribution in the following
sense:
\begin{itemize}
\item If $\psi_\lambda$ is the scaling in the plane by a factor $\lambda>0$, then the law of CLE$(\infty \to 0)$
is invariant under $\psi_\lambda$. Also CLE$(\infty \to \C)$ is invariant under $\psi_\lambda$. 
\item If $\phi_b$ is the translation in the plane by a complex number $b$, then
the law of CLE$(\infty \to b)$ is the image of of the law CLE$(\infty \to 0)$ under $\phi_b$.
The entire collection CLE$(\infty \to \C)$ is in fact invariant under $\phi_b$.
\end{itemize}
This property follows from coupling the CLEs in $n\U$ with $n \lambda \U + b$ (for a given 
$\lambda>0$ and $b \in \C$).

In the nested CLE in a domain $D$, the chains of loops to distinct points $\{b_1,b_2,\ldots,b_n\}$
are coupled so that the chains to $b_i$ and $b_j$ are the same until the loops of $b_i$ don't any more
surround $b_j$ and vice versa, after which the chains are conditionally independent.
This shows that the full-plane CLEs from $\infty$ to any of the points $\{b_1,b_2,\ldots,b_n\}$
can be coupled to have the same property.
Let us denote the restriction of CLE$(a \to \C)$ to those loops that disconnect $a$ and a point in $\{b_1,b_2,\ldots,b_n\}$
by CLE$(a \to \{b_1,b_2,\ldots,b_n\})$.
In the rest of the paper,
we will show that the law of CLE$(\infty \to \C)$ is fully M\"obius invariant
and by that result we can define the \emph{Riemann sphere nested CLE}, denoted by CLE$(\hat{\C})$,
whose law doesn't depend on the starting point.
One way to formulate this is that CLE$(a \to \{b_1,b_2,\ldots,b_n\})$ 
and  CLE$(b_1 \to \{a,b_2,\ldots,b_n\})$ have the same law, if we ignore the order of exploration of the loops.

\medbreak

We can  define $\nucle$ as the infinite intensity measure of CLE$(\infty \to 0)$, i.e.
the set of loops that surround the origin in the full-plane CLE,
and apply the same arguments that at the end of Section 2: The measure $\nucle$ is  scale-invariant,
invariant under $Q^{\to e}$, and its shape probability measure $P^{\rm cle}$ is invariant under $\tilde Q^{\to e}$. The previous coupling result shows that $P^{\rm cle}$ is the unique 
invariant shape distribution under $\tilde Q^{\to e}$, from which one can deduce that (up to a multiplicative constant) $\nucle$ is the unique scale-invariant  measure that is invariant under $Q^{\to e}$.

\subsection {Roadmap to reversibility of the full-plane CLE}

Let us now briefly sum up the measures on translation-invariant and scale-invariant random full-plane structures that we have defined at this point: 
\begin {itemize}
 \item[(i)] The nested CLE in the entire plane. When one focuses at the loops surrounding the origin, it has an intensity measure $\nucle$ that is, up to multiplicative constants, 
 the only scale-invariant measure that is invariant under $Q^{ \to e}$.
 \item[(ii)] The full-plane SLE$_{8/3}$ loop soup. 
  When looking at clusters and their boundaries that surround the origin, it defines 
 intensity measures $\nu$, $\nu^i$ and $\nu^e$. The former two are (up to multiplicative constants) the only invariant measures under $Q^{\to K}$ and $Q^{\to i}$ that are scale-invariant. 
 Furthermore $Q^{\to e} \nu^i = \nu^{e}$.
 \end {itemize}


We recall that (as opposed to the nested CLE) 
we already know at this stage that the full-plane SLE$_{8/3}$ 
loop soup is invariant under inversion
and that therefore the image of $\nu^i$ under $z \mapsto 1/z$ is $\nu^e$,
see Proposition~\ref{prop: measures under z mapsto 1/z}.

%
%


%

Our roadmap is now the following: In the next section, we are going to prove (building on the SLE$_\kappa$ description of CLEs and on various 
properties of SLE) the following fact:

\begin {proposition}
\label {nui=nue}
The two measures $\nu^e$ and $\nu^i$ are equal.
\end {proposition}

Let us now explain how Proposition~\ref{nui=nue} implies Theorem \ref {mainthm} (we defer the proof of the proposition to the next section):
First, note that Proposition \ref{nui=nue} 
implies immediately that 
 $Q^{\to e} \nu^e = Q^{ \to e} \nu^i =  \nu^e$, so that $\nu^e$ equal to is a constant times $\nucle$ (because it is invariant under $Q^{\to e}$). 
 Let us rephrase this rather surprising fact as a corollory in order to stress it:
 \begin {corollary}
 The kernels $Q^{\to e}$ and $Q^{\to i}$ have the same scale-invariant measures.
 \end {corollary}   
 As we know  that $\nu^e$ is the image of $\nu^i$ under $z \mapsto 1/z$, we can already conclude that $\nucle$ is in fact invariant under inversion. 
 
 \medbreak
 In order to prove Theorem \ref {mainthm}, it is  sufficient to prove the invariance in distribution under the map $z \mapsto 1/z$ of the 
nested family CLE$(\infty \to 0)$ of loops $(\gamma_j, j \in J)$.
 Indeed, on each of the successive annuli (in between $\gamma_j$ and $\gamma_{j+1}$), the conditional distribution (given the sequence $(\gamma_j)$) of the other loops of the same ``generation'' as $\gamma_{j+1}$ in the nested CLE (that are surrounded by $\gamma_j$ but by no other loop in between them and $\gamma_j$) is given by the outermost boundaries of loop-soup clusters in the annulus between 
 $\gamma_j$ and $\gamma_{j+1}$, conditioned to have no cluster that surrounds $\gamma_{j+1}$. This description of the conditional distribution is nicely invariant   
 under inversion (because the loop soups are), and this proves readily that the law of the entire nested CLE is invariant under $z \mapsto 1/z$. Since we already have translation-invariance and scale-invariance, this implies Theorem \ref {mainthm}.

 \medbreak 
 It now remains to prove that the law of 
 the nested family CLE$(\infty \to 0)$ of loops $(\gamma_j, j \in J)$
 is invariant under inversion. 
 Before explaining this, let us first make a little side-remark:
 Let us define the successive concentric annuli $(A_j, j \in J)$ in the nested CLE sequence
where $A_j$ denotes the annular region in between the loop $\gamma_j$ and its successor $\gamma_{j+1}$ (i.e. the next loop
in the sequence, inside of $\gamma_j$). As before, one can also define the  scale-invariant ``intensity measure'' on the set of annuli that we call 
$\nu^A$. The Markovian definition of the nested CLE sequence shows immediately that $\nu^A$ can be described from the product measure $\nucle \otimes \tilde P$ as follows:
\begin {itemize}
 \item $\tilde P$ is the law of the outside-most loop $\tilde \gamma$ that contains the origin in a CLE in the unit disc.
 \item Starting from a couple $(\gamma, \tilde \gamma)$, one defines the annulus $A$ that is between $\gamma$ on the one hand (that is therefore the outer loop of the annulus) and $\phi_\gamma (\tilde \gamma)$ where $\phi_\gamma$ is the conformal map from the unit disc onto the inside of $\gamma$ such that $\phi(0)=0$ and $\phi'(0) > 0$.
\end {itemize}
But, one can observe that almost surely, in a nested CLE sequence, only one annulus between successive loops (that we call $A(1)$) does contain the point $1$. Hence, $\nu^A$ restricted to those annuli that contain $1$ is a probability measure, and this probability measure $P^{A, 1}$ is the law of $A(1)$. 

One can apply a similar construction to the full-plane SLE$_{8/3}$ loop soup, focusing on the annuli that are in between a loop $\gamma_j^i$ and the next outer boundary $\gamma_{j+1}^e$. It follows (using the fact that $\nucle$ is equal to a constant times $\nu^e$) that, up to a multiplicative constant (corresponding to the probability that a given point is in between to such loops) the measure $\nu^A$ describes also the 
intensity measure of such annuli in the full-plane loop soup. We can now use the inversion-invariance of the full-plane loop soup and the fact that $\nu^i=\nu^e$, to 
conclude that  
 that $\nu^A$ can also be constructed from inside-out as follows: Define the inner loop via $\nucle$ and choose the outer loop by sampling a CLE in the outside of the inner loop, and take the innermost loop in this CLE. From this, it follows that $\nu^A$ is invariant under $z \mapsto 1/z$, and therefore $P^{A, 1}$ too. 

In order to prove the inversion invariance of the law of the entire nested sequence $(\gamma_j, j \in J)$, we proceed in almost exactly the same way, except that 
we now focus on the joint law of the $2n_0$ loops ``closest'' to $1$ in the sequence: 
 Let us index the 
 loops by $1/2 + \Z$ in such a way that the point $1$ is in between the two successive loops $\gamma_{-1/2}$ and $\gamma_{1/2}$. 
  Let us choose any integer $n_0 \ge 1$, and look at the random family consisting of the $2n_0 $ loops nearest to the point $1$ i.e. 
 $$ \Gamma^{n_0}:= (\gamma_{-n_0+1/2}, \ldots, \gamma_{-1/2}, \gamma_{1/2}, \ldots , \gamma_{n_0 -1/2}).$$

One way to describe the law of of $\Gamma^{n_0}$ is to start with the infinite measure on $2n_0+2$-tuples of loops obtained by defining the first one according to 
the infinite scale-invariant measure $\nucle$ and then to use $2n_0 -1$ times the Markov kernel $Q^{\to e}$ in order to define its successors, and then to restrict this infinite measure to the set of $(2n_0 +2)$-tuples of loops such that $1$ is in between the two middle ones.

Exactly the same arguments as for $n_0=1$ then show that $\Gamma^{n_0}$ can alternatively 
be defined inside out, so that the law of $\Gamma^{n_0}$ ---
and therefore of the entire sequence, as this holds for all $n_0$ ---
is invariant under $z \mapsto 1/z$.

\subsection {Remarks on the Markov chain of annular regions}

Note that the previous annuli measure $\nu^A$ is scale-invariant; we can therefore define its associated shape probability measure $P^A$.
We will denote by $m (A)$ the unique $m <1$ such that $A$ can be mapped conformally onto $\{ z \, : \, m < |z| < 1 \}$. 

We can note that with the description of $\nu^A$ via $\nucle \otimes \tilde P$, the modulus $m(A)$ of the annulus is fully encoded by 
$\tilde \gamma$ (as it is the modulus of 
the part of the unit disc that is outside of $\tilde \gamma$). In particular, restricting $\nu^A$ to the set of annuli of a certain modulus (say for $m(A) \in (m_1, m_2)$), one obtains a scale-invariant measure on annuli described by the previous method from  $\nucle \otimes \tilde P^{m_1, m_2}$ (where $\tilde P^{m_1, m_2}$ means the probability $\tilde P$ restricted to those loops that define an annulus with modulus in $(m_1, m_2)$). 
In other words, the ``marginal measure'' on the outside of such annuli is just a constant $c(m_1, m_2)$ times $\nucle$, and its shape probability is still $P^{\rm cle}$. 

But our
CLE symmetry result shows that it is also possible to view the nested CLE sample as being defined iteratively from inside to outside. Furthermore, the modulus of 
an annulus $A$ surrounding the origin and of $1/A$ are identical. Hence, it follows immediately that the marginal measure of the inside loop of the annulus (restricted to those  with modulus in $(m_1, m_2)$ is also $c(m_1, m_2) \nucle$ and that its shape probability measure is again $P^{\rm cle}$. 

Hence, it follows  that:

\begin {proposition}
 If we define the Markov kernel $Q^{\to e, (m_1, m_2)}$ just as $Q^{\to e}$ except that we restrict ourselves to those jumps that correspond to an annulus with modulus in $(m_1, m_2)$, then $\nucle$ is again (up to a multiplicative constant) its unique invariant measure. 
 \end {proposition}

The following two extreme cases are of course particularly worth stressing:
\medbreak

\noindent
(a) \quad When $m_1=0$ and $m_2$ gets very small, we see on the one hand by standard distortion estimates that the shapes of the inside loop and of $\tilde \gamma$ become closer and closer, and on the other hand, that shape of the inside loop is always described by the shape of $\nu$. Hence, this leads to the following description of the CLE shape distribution $P^{\rm cle}$:
\begin {corollary}
 Consider a CLE in the unit disc and let $\hat \gamma$ denote the outermost CLE loop that surrounds the origin, and let $m$ denote the modulus of the annulus between $\hat \gamma$ and the unit circle. Then, the law of the shape of $\hat \gamma$ conditioned on the event $\{ m < \eps \}$ does converge to the shape distribution $P^{\rm cle}$ as $\eps \to 0$.
\end {corollary}
Loosely speaking, the very small loops in a simple non-nested CLE describe the stationary shape $P^{\rm cle}$.

\medbreak
\noindent
(b) \quad
 When $m_2=1$ and $m_1$ is very close to one, then when $m (A) > m_1$, the inside and the outside loop are (in some conformal way) conditioned to get very close 
to each other. Again, both the outer and the inner shape are always described by $P^{\rm cle}$. It is actually possible to make sense of the limiting kernel 
$Q^{\to e, (m_1, 1)}$ as $m_1 \to 1$. This gives a scale-invariant measure on ``degenerate'' annuli where the inside and outside loops intersect, 
and where the marginals of the shape measure for both the inside and the outside loops are described via $P^{\rm cle}$. In the case where $\kappa = 4$, this is very closely related to 
the conformally invariant growing mechanism described in \cite {WW} and to work in progress, such as e.g. \cite {SWW}.

\section {Proof of $\nu^i = \nu^e$}

\subsection {Exploring (i.e. dynamically constructing) loop-soup clusters}

In the present subsection, we review some ideas and tools introduced in \cite {ShW} about simple CLEs, and discuss some consequences in the present setup. 

In the sequel, we will say that a conformal transformation $\varphi$ from $\HH$ onto a subset of the Riemann sphere defines a ``marked domain'' (as it gives information about the domain $\varphi (\HH)$ as well as the image of some marked points, say of $i$ and $0$). 

In \cite {ShW} (see also \cite {WW}), it has been studied and explained how to construct a simple CLE in a simply connected domain $D$ from a Poisson point process of SLE bubbles. 
Let us briefly and somewhat informally recall this construction in the case where $D$ is the disc of radius $R$ around the origin. First, note that when one continuously 
moves from $-R$ along the segment $[-R, 0]$, one encounters loops of a CLE one after the other. 
The CLE property loosely speaking states that if one discovers the whole loop 
as soon as one bounces it, then the law of the loops in the remaining 
to be explored domain is still that of a CLE in that domain. 
This leads to the 
fact that the loops that one discovers can be viewed as arising from a Poisson point process of boundary bubbles (that turn out to be SLE$_\kappa$ bubbles)-- see \cite {ShW} for details. 
And indeed, it is in fact possible to construct a CLE, when starting from a Poisson point process of such SLE$_\kappa$ bubbles. More precisely, one 
first defines the infinite measure $\mu$ on SLE$_\kappa$ loops in the upper half-plane that touch the real line only at the origin. This is the appropriately scaled 
limit when $\epsilon \to 0$ of the law of an SLE$_\kappa$ from $0$ to $\epsilon$ in the upper half-plane. Then, one considers a Poisson point process of such bubbles with intensity $\mu$ and from this Poisson point process, one can construct all the simple non-nested CLE loops that for instance intersect a circle $C$ that surrounds the origin 
 (one just replaces the previous segment by the path $[-R, -r]$, where $-r \in C$ to which one then attaches the circle $C$ (along which one then moves continuously and discovers all the CLE loops that it intersects). 
Note also that if one discovers a loop that surrounds the circle $C$ on the way, then it is possible to continue the exploration in its inside if one is considering a nested CLE. In this way, one can constructs in fact all the loops that intersect $C$ in a {\em nested} CLE. See \cite {ShW} for details.

\begin{figure}[ht]
\begin{center}
\includegraphics[scale=0.4]{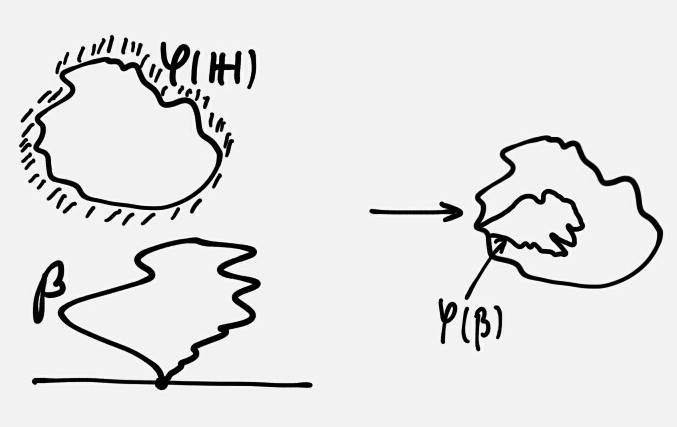}
\end{center}
\caption {From $\varphi$ and $\beta$ to the loop $\varphi (\beta)$}
\end {figure}

This procedure shows the existence of a measure $\rho^{C}_R$ on the set of marked domains
so that the image of the measure $\rho^{C}_{R} \otimes \mu$ under the map $(\varphi, \beta) \mapsto \varphi (\beta)$, and restricted to those pairs for which 
 $\varphi (\beta)$  intersects $C$, is equal exactly to the intensity measure of nested CLE loops in the disc of radius $R$, restricted to those that intersect $C$. In other words, for any set $A$ of loops that intersect $C$, if $\Gamma_R$ denotes a nested CLE in the disc of radius $R$, then
\begin {equation}
 \label {relation1}
 \EX \left( \sum_{\gamma \in \Gamma_R} \ind_{ \gamma \in A} \right) = (\rho^{C}_{R} \otimes \mu ) ( \{ ( \varphi, \beta) \, : \, \varphi (\beta) \in A \} ).
 \end {equation}
The previously described convergence and coupling arguments on nested CLEs when $R \to \infty$ in fact readily show that the previous statement also holds in the full-plane CLE setting (just continuing independently the exploration inside each of the discovered loops). More precisely:

\begin {lemma}
There exists a measure $\rho^C$, so that if $\Gamma$ denotes a full-plane CLE, (\ref {relation1}) holds when one replaces $(\Gamma_R, \rho^C_R)$ by 
$(\Gamma, \rho^C)$.
\end {lemma}

\medbreak
Let us now consider a full-plane SLE$_{8/3}$ loop soup instead and its clusters. One can then apply almost the  same argument as in the nested CLE to obtain the following statement: Let $\Gamma_e$ denote the set of {\em outer boundaries} of clusters in this full-plane SLE$_{8/3}$ loop soup. Then:
\begin {lemma}
 There exists a measure $\tilde \rho_C$ on the set of marked domains, so that (\ref {relation1}) holds when one replaces $(\Gamma_R, \rho^C_R)$ by $(\Gamma_e, \tilde \rho^C)$.
\end {lemma}
The two little modifications that are needed in order to justify this fact are:
\begin {itemize}
 \item That one needs to replace the measure on SLE$_\kappa$ bubbles (i.e. boundary-touching loops) by a measure on 
``boundary-touching clusters''. 
The existence and construction of this measure is obtained in exactly the same way as the existence 
and construction of the CLE bubble measure in Sections~3 and 4 of \cite {ShW}.
 \item That when one encounters a cluster that surrounds (or intersects) $C$, then one continues to explore inside all connected components of its complement that intersect the circle $C$ independently.
\end {itemize}

Recall that the full-plane SLE$_{8/3}$ loop soup is invariant under any M\"obius transformation of the Riemann sphere.  
Hence, we can reformulate the previous property after applying the conformal transformation $z \mapsto 1/ (z-z_0)$. 
We therefore obtain, for each point $z_0$ in the plane with $z_0 \not=0$, and any (small) circle $C$ surounding $z_0$, a description of the 
measure on the set of those boundaries of SLE$_{8/3}$ clusters, that intersect  $C$ and separate $z_0$ from the rest of the cluster (this corresponds to the fact that the previous description was describing the ``outer boundaries'' of the clusters, which are those that separate the cluster from infinity). In the sequel, we shall in fact in particular focus on those loops that do disconnect the origin from infinity (i.e. the $\gamma_j^e$ and $\gamma_j^i$ loops). Among those, the previous procedure describes/constructs:
\begin {itemize}
 \item The loops $\gamma^e_j$ that do not surround $z_0$ and intersect $C$.
 \item The loops $\gamma^i_j$ that do surround $z_0$ and intersect $C$. 
\end {itemize}

\begin{figure}[ht]
\begin{center}
\includegraphics[scale=0.5]{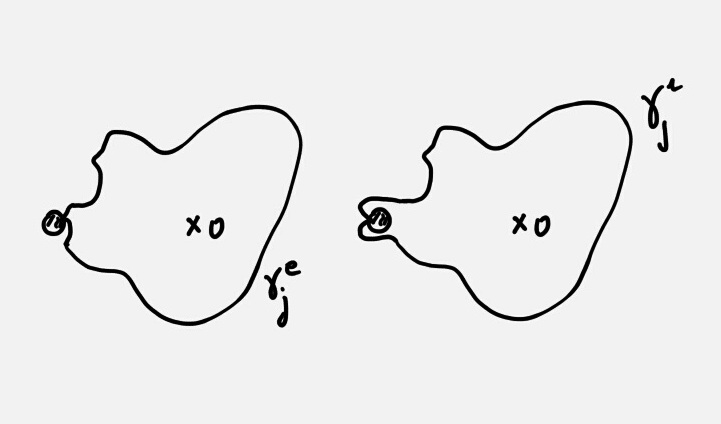}
\end{center}
\caption {The two type of configurations that one can discover}
\end {figure}

\medbreak
In all the remainder this section, when we will mention  ``the $\epsilon$-neighborhood of $z_0$'' in the plane (for $z_0 \not= 0$), this will always mean 
the disc of radius $|z_0| \sinh \epsilon$ around $z_0 \times \cosh \epsilon$ (i.e. with diameter $[z_0 e^{-\epsilon}, z_0 e^{\epsilon}]$). 
In particular, we see that with this definition, (i) when $\epsilon$ is very small, the $\epsilon$-neighborhood of $1$ is quite close to the Euclidean $\epsilon$-neighborhood of $1$.
 Furthermore, (ii) for all $z_0 \not=1$, the $\epsilon$-neighborhood of $z_0$ is equal to the image under $z \mapsto z_0 z$ of the $\epsilon$-neighborhood of $1$, and (iii), the  $\epsilon$-neighborhood of $z_0$ is invariant under the inversion $z \mapsto z_0^2/z$. 

Similarly, we will denote $d(z, K)$ to be the largest $r$ such that $K$ remains disjoint of the $r$-neighborhood of $z$. 

\medbreak

If we apply our previous analysis to the case where the circle $C$ is the boundary of the $\eps$-neighborhood of $z_0$, we therefore obtain the existence of a measure $\tilde \rho^{z_0, \eps}$ 
on marked domains (here the marked domain is simply connected in the Riemann sphere but not necessarily simply connected in $\C$) such that:
\begin {itemize}
 \item The measure $\nu^e$ restricted to those loops that intersect $C$ and do not surround $z_0$, is equal to the image of the measure $\tilde \rho^{z_0, \epsilon} \otimes \mu$ under the mapping $(\varphi, \beta) \mapsto \varphi (\beta)$, when restricted to those loops that do separate $0$ from infinity, and do not surround $z_0$.
 \item The measure $\nu^i$ restricted to those loops that intersect $C$ and do surround $z_0$, is equal to the image of the measure $\tilde \rho^{z_0, \epsilon} \otimes \mu$ under the mapping $(\varphi, \beta) \mapsto \varphi (\beta)$, when restricted to those loops that do separate $0$ from infinity, and do surround $z_0$.
\end {itemize}
We can also note that inversion-invariance of the loop-soup picture shows that in each of these two statements, it is possible to replace $\tilde \rho^{z_0, \epsilon}$ by its image $\hat \rho^{z_0, \epsilon}$ under 
$z \mapsto z_0^2/z$ (it just corresponds to exploring/constructing the image of the loop-soup cluster boundaries under this map).

\medbreak
Our goal is now to build on these constructions in order to show that the measures $\nu^e$ and $\nu^i$ are very close when $\eps \to 0$. 
Let us denote by $V_\eps^e (\gamma)$ (respectively $V_\eps^i (\gamma)$) to be the the set of points that are  in the $\epsilon$-neighborhood of a loop $\gamma$ {\em and } lie outside of it (respectively inside of it).
Clearly, for each loop $\gamma$, 
$$ \int_{\C} \de^2 z \, \ind_{\{ z \in V_\eps^e (\gamma) \} }  / | V_\eps^e(\gamma) |  = 1 ,$$
(where $|V|$ denotes here the {\em Euclidean} area of $V$) so that it is possible to decompose the measure $\nu^e$ as follows:
\begin {eqnarray*}
\nu^e ( F(\gamma) ) 
&=& \nu^e \left( F(\gamma) \int_{\C} \de^2 z \, \ind_{\{ z \in V_\eps^e (\gamma) \} }   / | V_\eps^e(\gamma) | \right)
\\
&=& 
\int_\C \de^2 z \, (\tilde \rho^{z,\eps} \otimes \mu)^e \left(  F( \gamma) / | V_\eps^e(\gamma) |  \right),
\end {eqnarray*}
where  
$(\tilde \rho^{z,\eps} \otimes \mu)^e$ denotes the measure on $\gamma= \varphi (\beta)$ restricted to the configurations where one constructs
$\gamma$ ``from the outside'' a loop surrounding the origin (i.e. $z$ lies on the outside of this loop). 
Rotation and scale-invariance shows that 
$$(\tilde \rho^{z,\eps} \otimes \mu)^e \left(  F( \gamma) / | V_\eps^e(\gamma) |  \right) 
= (\tilde \rho^{1,\eps} \otimes \mu)^e \left(  F( z\gamma) / ( |z|^2 | V_\eps^e(\gamma) |)   \right)
.$$
We can now interchange again the order of integration, which leads to 
\begin {equation}
\label {nue}
\nu^e ( F (\gamma) ) = (\tilde \rho^{1, \eps} \otimes \mu)^e \left( {\tilde F (\gamma) } / {| V_\eps^e(\gamma) |}  \right) 
\end {equation}
where $\tilde F (\gamma) = \int_{\C} \de^2 z \, F ( z \gamma)/ |z|^2 
$ (provided that $F$ is chosen so that the above integrals all converge, for instance if it is bounded and its support is included in the 
set of loops that wind around the origin and stay in some fixed annulus $D(0, r_2) \setminus D(0, r_1)$).  

\medbreak

Using inversion, we get the similar expression for $\nu^i$, 
\begin {equation}
 \label {nui}
\nu^i ( F (\gamma) ) = (\tilde \rho^{1, \eps} \otimes \mu)^i \left(  { \tilde F (\gamma) } / { | V_\eps^i(\gamma) | } \right), 
\end {equation}
where this time,  the notation $(\tilde \rho^{1, \eps} \otimes \mu)^i$ means that we now consider the loops $\gamma = \varphi (\beta)$ that surround both the origin and $y_0$.

Let us stress that the two identities (\ref {nue}) and (\ref{nui}) hold for all $\eps$.

\medbreak

In order to explain what is going to follow in the rest of the paper, let us now briefly (in one paragraph) outline the rest of the proof: 
We  want to prove that $\nu^e (F) = \nu^i (F)$ for a sufficiently wide class of functions $F$ and to deduce from this that $\nu^e = \nu^i$.
In order to do so, we will 
show that in the limit when $\eps \to 0$, the two right-hand sides of (\ref {nue}) and (\ref {nui}) above behave similarly. 
For this, we will use the results and ideas of \cite {LR} on the Minkowski content of SLE paths, that loosely speaking  will show that when $d = 1 + \kappa/8$, for 
$\nu^e$ and $\nu^i$ almost all loop $\gamma$, there exists a deterministic sequence $\eps_n $ that tends to $0$  such that 
(this will be Lemma \ref {minkovski}),
$$  | V_{\eps_n}^i(\gamma) |  \sim  | V_{\eps_n}^e (\gamma) | \sim \eps_n^{2-d} L(\gamma),$$
as $n \to \infty$,
where $L(\gamma)$ is a positive finite quantity related to the ``natural'' (i.e. geometric) time-parametrization of the SLE loop.
We will rely on the one hand on this fact, and on the other hand, on the fact that when $\eps \to 0$, 
the measure $\eps^{d-2} (\tilde \rho^{1, \eps} \otimes \mu)^e$  converges to the same measure $\lambda$ on loops that pass through $1$ and separate the origin from infinity as 
$\eps^{d-2} (\tilde \rho^{1, \eps} \otimes \mu)^i$ does. Basically (we state the following fact as it may enlighten things,  even though we will not explicitly prove it because it is not needed in our proof), one has an expression of the type 
$$ \nu^i ( F (\gamma) ) = \lambda (\tilde F (\gamma) / L(\gamma)) = \nu^e (F (\gamma)).$$
Our proof will be based on a coupling argument that enables to compare the right-hand sides of (\ref {nue}) and (\ref {nui}).
Another fact
that it will be handy to use in the following steps is the reversibility of SLE$_\kappa$ paths for $\kappa \in (8/3, 4]$. There exist now several different proofs of this result first proved by Dapeng Zhan in \cite {Zhan}, see for instance \cite {MS1,WW2}.

\medbreak

We now come back to our actual proof of the fact that $\nu^e=\nu^i$, and state the following lemma; we postpone its 
proof  to the next and final subsection of the paper, as it involves somewhat different arguments (and results of \cite {LR} on the Minkowski content of chordal SLE paths). 
\begin {lemma}
\label {minkovski}
There exists a sequence $\eps_n$ that tends to $0$, such that for $\nu^e$ almost every loop, there exists a finite positive $L(\gamma)$ such that 
 \begin {equation}
  \label {equiv}
 \lim_{n \to \infty} \eps_n^{d-2} | V_{\eps_n}^e (\gamma) | =
 \lim_{n \to \infty} \eps_n^{d-2} | V_{\eps_n}^i (\gamma) | = L(\gamma).
 \end {equation}
\end {lemma}
Note that this is equivalent to the fact that almost surely, 
for any exterior boundary of a cluster that surrounds the origin in a full-plane loop-soup sample, (\ref {equiv}) holds. 
One could also derive the more general statement (taking the limit when $\eps \to 0$ instead of along some particular sequence $\eps_n$) but Lemma \ref {minkovski} will be sufficient 
for our purpose. 

\medbreak

Let us now explain how to use this lemma in order to conclude our proof. 
We will  couple $(\tilde \rho^{z,\eps} \otimes \mu)^e$ with 
$(\hat \rho^{z,\eps} \otimes \mu)^i$. 
Note first that because of inversion-invariance of the full-plane SLE$_{8/3}$ loop-soup, the total masses of the two measures $(\tilde \rho^{z,\eps} \otimes \mu)^e$ and $(\hat \rho^{z,\eps} \otimes \mu)^i$ are identical. In fact, these masses decay like a constant times $\eps^{2-d}$ as $\eps \to 0$ (but for what follows, it will be enough to note that they are bounded by a constant time $\eps^{2-d}$). 

Let us now describe how we define our coupling: 
First, for a choice of a marked domain by $\tilde \rho^{z, \eps}$ in our first measure, we consider the one obtained by the inversion $y \mapsto z^2 / y$ for the ``sample''
of $\hat \rho^{z, \eps}$. 
Hence, after mapping our marked domains onto the unit disc in such a way that $\infty$ and the origin are mapped onto two real symmetric points $-a$ and $a$ respectively, 
we have to compare/couple the following two measures on bubbles:
\begin {itemize}
 \item The measure on SLE$_\kappa$ bubbles in the unit disc, rooted at some point $e^{i \theta}$ and restricted to the set of bubbles that surround $a$ and not $-a$
 \item The measure on SLE$_\kappa$ bubbles in the unit disc, rooted at the point $-e^{i \theta}$ and restricted to the set of bubbles that surround $-a$ and not $a$.
\end {itemize}
By symmetry, these two measures have again the same mass. The goal is now to couple two loops $\beta^1$ and $\beta^2$ (each defined under these SLE$_\kappa$ bubble measures)
in such a way that when $a$ is small, for most realizations of $\beta^1$ and $\beta^2$, the two loops are in fact very similar in the neighborhood 
of the origin (in the disc of radius $\sqrt {a}$, say). This would indeed then imply that when mapped back onto the marked domain, the loops are very close, except in a small neighborhood of $z$ (and therefore very close everywhere).

Let us first sample progressively  a part of $\beta_1$ starting from its root $e^{i \theta}$.
One natural way to encode this exploration in the present setting is to always map the complement of the curve in the unit disc back to the unit disc,
in such a way that the two point $-a$ and $a$ are mapped onto two symmetric real values $-a_t$ and $a_t$. This fixes the conformal transformation (note
also that $a_t$ is increasing with time,  which can enable to use $a_t$ to define a convenient time-parametrization). The tip of the curve is 
mapped onto some $e^{i \varphi_t}$ while the target (i.e. one of the images of $e^{i\theta}$) is mapped onto some $e^{i \theta_t}$ on the unit circle.   
We are interested in the time 
$T$ (when it exists), which is the first time at which $ e^{i\varphi_t} = - e^{i \theta_t}$. Then at this time, after mapping back the complement of the already discovered part of $\beta^1$ to the unit disk, the remaining to be discovered path is an SLE$_\kappa$ from the random point $b:= e^{i \varphi_T}$ to $-b$ that we restrict to the event that it disconnects $-a_T$ from $a_T$. 

We shall see a little later (it will convenient to explain this in the next subsection, together with some other result proved there) that:
\begin {lemma}
\label {thealemma}
Consider the SLE bubble measure rooted at $e^{i \theta}$ and restricted to those that disconnect $-a$ from $a$.
 When $a \to 0$, the proportion of loops such that $T < \tau_{a^{3/4}}$ tends to one. 
 Here $\tau_r$ is the hitting time of the circle around $r$ in the previous parametrization.
%
%
\end {lemma}

By symmetry, for the begining of the loop starting from $- e^{i \theta}$, we can use exactly the symmetric path (with respect to the origin), so that at the same time $T$, the configuration is exactly symmetric one (see Figure \ref {couplingidea}). 
In both cases, modulo the conformal transformation corresponding to the curve up to time $T$, the remaining curve is just an SLE from $b$ to $-b$ in the unit disc, restricted to 
those configurations that separate $-a_T$ from $a_T$. Then, in our coupling we can use the very same SLE sample (i.e. using reversibility of the SLE) in both cases for this remaining SLE. 

\begin{figure}[ht]
\label {couplingidea}
\begin{center}
\includegraphics[scale=0.5]{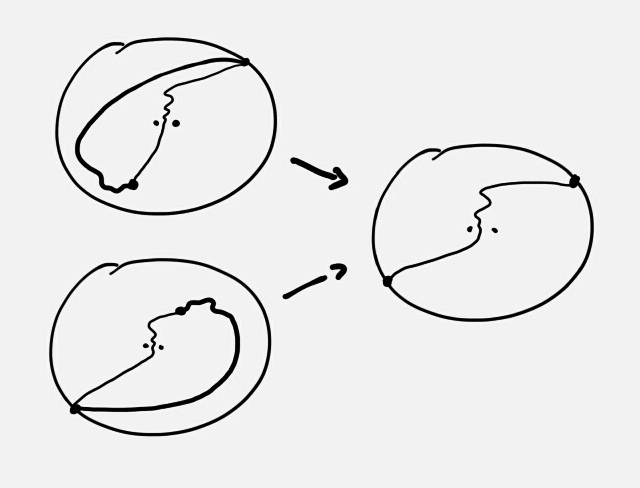}
\end{center}
\caption {The coupling idea}
\end {figure}

This therefore defines a coupling $\bar \rho^{1,\eps}$ of the two measures 
$(\tilde \rho^{z,\eps} \otimes \mu)^e$ and $(\hat \rho^{z,\eps} \otimes \mu)^i$, in such a way that $(\gamma^1, \gamma^2)$ are defined each under these two measures (with same total mass) and are very close (for a large proportion of their masses).
 Furthermore, with this coupling, 
the quantities of the type $V_\eps^i (\gamma)$ and $V_\eps^e (\gamma)$ are very close as well (when $\eps$ is small) for $\gamma^1$ and $\gamma^2$ (because except on a small piece very close to $z_0$, the two loops are the conformal images of the same SLE$_\kappa$ path under almost the same conformal transformation). 

Let 
$$\overline L (\gamma) := \sup_{n} \epsilon_n^{d-2} \max ( | V_{\epsilon_n}^i (\gamma) |, |V_{\epsilon_n}^e (\gamma) | ). $$ 
We know that this quantity is finite for almost all loops that we are considering (because of Lemma \ref {equiv}).

Let $R>0$, and choose $\Phi_R: [0, \infty) \to [0,1]$ to be some smooth function that is equal to 1 on $[2/R, \infty)$ and to $0$ on $[0, 1/R]$. We let 
$\tilde \Phi_R (\gamma) = \Phi_R (\overline L(\gamma))$.

We can then simply estimate 
$$ \left|  \nu^e ( F (\gamma) \tilde \Phi_R (\gamma) ) 
- \nu^i ( F (\gamma) \tilde \Phi_R (\gamma) )  \right|
\le \eps_n^{d-2} \bar \rho^{1, \eps_n}   \left( \left|  \frac { \tilde F (\gamma^1) \tilde \Phi_R (\gamma^1) }{ \eps_n^{d-2} | V_{\eps_n}^e(\gamma^1) |} 
- \frac { \tilde F (\gamma^2) \tilde \Phi_R (\gamma^2) }{ \eps_n^{d-2} | V_{\eps_n}^i(\gamma^2) |} \right|  \right) .
$$
By dominated convergence --- the integrand is bounded, the total mass remains bounded and the integrand 
 gets close to zero except on a smaller and smaller portion of the space --- we see that this tends to $0$ as $n \to \infty$.


Since this holds for all $F$ and $R$, we conclude that in fact $\nu^i = \nu^e$.

\subsection {Minkowski content and symmetry}

This final subsection will mostly devoted to the proof of Lemma \ref {minkovski}. 
In fact,  we will deduce it from the following similar lemma concerning chordal SLE paths (and not loops). Suppose that $J$ is non-negative continuous compactly supported function from $\overline \HH$ to $\R$ that is equal to $0$ on a neighborhood of the origin, and that $\beta$ is a simple curve from $0$ to infinity in $\HH \cup \{ 0\}$ (starting at $0$, and tending to infinity at the other end). One can then define 
$H^+ (\beta)$ and $H^-(\beta)$, the two connected components of $\HH \setminus \beta$ that respectively lie to ``its right'' and to ``its left'', and define 
$$ v_{\eps, J} (\beta)= \int_{\HH} \de^2 z \, J(z) \ind_{ d(z,\beta) < \eps},   \quad v_{\eps, J}^ + (\beta) 
  = \int_{H^ + ( \beta)}  \de^2 z \, J( z) \ind_{d(z,\beta) < \eps}$$
and $v_{\eps, J}^- (\beta) = v_{\eps, J} (\beta) - v_{\eps, J}^+ ( \beta)$.

\begin {lemma}
\label {minkovski2}
Suppose that $\kappa \in (8/3, 4]$.  There exists a sequence $\eps_n$ that tends to $0$, such that for all non-negative continuous compactly supported function $J$ defined on $\HH$, one has for almost  every chordal SLE$_\kappa$ path $\beta$, 
$
v_{\eps_n, J}^+ (\beta) \sim v_{\eps_n, J}^- (\beta) 
$
as $n \to \infty$.
\end {lemma}

Before proving this lemma, let us first explain how one can deduce Lemma \ref {minkovski} from it:

\begin {enumerate}
\item 
The arguments of \cite {LR}, see analogous statement in Theorem 3.1 of \cite{LR} (preprint version),
 go through with basically no modification in order to prove that for all given $J$ as above, for $d= 1+ \kappa/8$, and 
 for almost all SLE$_\kappa$ curve $\beta$,
 $$  \lim_{\eps \to 0} \eps^ {d-2} v_{\eps, J} (\beta) =  L_{J} (\beta), $$
 where $L_J (\beta)$ is positive on the set of curves that pass through the support of $J$. 
 The only difference with \cite {LR} is the presence of the weighting $J$ (but this is in fact also treated there in the context of the ``covariant'' properties
of the Minkowski content).
 Combining this with Lemma \ref {minkovski2} ensures that almost surely, 
 $$ 
v_{\eps_n, J}^+ (\beta) \sim v_{\eps_n, J}^- (\beta) \sim v_{\eps_n, J} (\beta)/2 \sim\eps_n^{2-d}  L_J (\beta) /2 $$
as $n \to \infty$.

\item 
Consider now instead of a chordal SLE path, an SLE bubble defined under the infinite measure $\mu$ (its description as an SLE excursion measure
is for instance described in \cite {ShW}), and keep a (compactly supported) function $J$ as before. 
Let $2r$ denote the distance between the origin and the support of $J$. We can discover a first bit of the SLE bubble until it reaches the circle of radius $r$ around the origin (if it does so, which is the case for a set of bubbles with finite $\mu$-mass). If we renormalize the measure appropriately, we can say that ``conditionally'' on this first part, the law of the remaining-to-be-discovered part of the bubble is a chordal SLE in the slit upper half-plane joining the tip of the already-discovered part to $0$. If we map this onto the upper half-plane and the tip and $0$ respectively onto $0$ and $\infty$, we see readily that the result stated in Step 1 (applied to a random function $\tilde J$ that depends on $J$ and the first part of the bubble) actually yields that for any $J$ as before, for $\mu$-almost every bubble $\beta$, 
 $$ 
v_{\eps_n, J}^i (\beta) \sim v_{\eps_n, J}^e (\beta) \sim \eps_n^{2-d} L_J (\beta) /2 
$$ 
as $n \to \infty$,
where the superscripts $i$ and $e$ now stand for the ``interior'' and the ``exterior'' of the bubble $\beta$.

\item
We are now going to restrict ourselves to the study of loops $\gamma$ in the plane that surround the origin and stay confined in a given annulus $A=\{ z \, : \,  r_1 < |z| < r_2 \}$. 
Let us define the functions $f^+ $ and $f^-$ in the plane such that $z \mapsto f^+(z) = f^- (-z)$  is equal to $1$ on $\{ z \, : \, \Re (z) > r_1 /2 \}$, to $0$ on $\{ z \,  : \, \Re (z) < -r_1 /2  \}$ and to $1/2 +  (\Re (z))/ r_1)$ in the medial strip 
$\{ \Re (z) \in [-r_1/2, r_1/2]\}$. Note that $f^+ + f^- = 1$. We claim that in order to prove Lemma \ref {minkovski}, it is sufficient to show that for $\nu^e$ almost all loop $\gamma$, 
$$
 \lim_{n \to \infty} \eps_n^{d-2} | V_{\eps_n, f^+}^e (\gamma) | =  \lim_{n \to \infty} \eps_n^{d-2} | V_{\eps_n, f^+}^i (\gamma) | = L_{f^+} (\gamma)
$$
as $n \to \infty$, where $L_{f^+}$ is positive on the set of loops under consideration (that wind around the small disc of radius $r_1$). 
Indeed, by symmetry, the same result is then true for $f^-$, and adding the two contributions (for $f^+$ and $f^-$), we get that 
$$
 \lim_{n \to \infty} \eps_n^{d-2} | V_{\eps_n}^e (\gamma) | =
 \lim_{n \to \infty} \eps_n^{d-2} | V_{\eps_n}^i (\gamma) | = (L_{f^+} + L_{f^-})(\gamma).
$$

\item 
In order to prove this statement with $f^+$, we can consider a full-plane SLE$_{8/3}$ picture, and discover all loop-soup clusters ``from the outside'', by first discovering (at once) all clusters that intersect the circle of radius $r_2$, and then ``moving inwards'' along the segment $[-r_2, -r_1]$. In this way, we can represent the measure 
$\nu^e$ (restricted to the set of loops that wind around the origin in the annulus $A$) as the image of a measure $(\rho \otimes \mu)$ under a map $(\varphi, \beta) \mapsto \varphi (\beta)$ as before, where $\varphi (0)$ is always on the segment $[-r_2, -r_1]$ which is a point in the neighborhood of which  $f^+$ vanishes. 
Applying the result of Step 2 to the function 
$$  | (\varphi^{-1})' (\cdot) |^{2-d} f^+ (\varphi^{-1} ( \cdot )), $$  
we then immediately get the statement that is needed in Step 3 (the fact that this function explodes in the neighborhood of some part of the real line does not matter, as the 
bubble almost surely remains at positive distance from those parts).
\end {enumerate}

\medbreak

It then remains to explain how to adapt the ideas of \cite {LR} in order to prove Lemma \ref {minkovski2}, and to also explain how to derive Lemma \ref {thealemma}. 
The results on the Minkowski content of chordal SLE in \cite {LR} are based on the one hand on the fact that 
when $z_0$ is a point in the upper half-plane, then one has a very good control on the probability that the SLE 
intersects the $\eps$-neighborhood of $z_0$: More precisely, if one maps the upper half-plane onto the unit disc 
by the conformal transformation $\phi$ such that $\phi (z_0)=0$ and $\phi (\infty) = -1$, then one can view 
the (chordal) SLE from $\phi(0)$ to $-1$ in the unit disc as a Loewner chain in the unit disc, viewed from the origin. 
If one re-parametrizes it via the log-conformal radius $u=u(t)$ seen from $0$, one gets easily that the argument $\theta_u$ 
of $\phi_t ( W_t)$ (where $\phi_t$ maps $\HH \setminus \beta[0,t]$ onto $\U$ with $\phi_t(\infty) = -1$ and $\phi_t (z_0) =0$) evolves according to the SDE
\begin {equation}
 \label {theSDE}
\de \theta_u = \sqrt {\kappa} \,\de B_u + \frac {4-\kappa}{2} \tan \frac {\theta_u}{2} \,\de u
\end {equation}
where $B$ is an standard one-dimensional Brownian motion.
The log-conformal radius of $\HH \setminus \gamma [0,\infty)$ seen from $z_0$ can therefore expressed as the 
life-time of the diffusion $\theta$  (the time it takes before hitting $-\pi$ or $\pi$).

One can note that in this framework, it is also easy to see on which side of $\gamma$ the point $z_0$ is. It will be in $H^+$ if and only if 
the diffusion $\theta$ hits $-\pi$ before $\pi$. As a consequence, this setup and the same arguments can be used to see 
 that when $\eps \to 0$, 
\begin {equation}
\label {thirdeq}
 \PR ( z_0 \in H^+ (\beta) \, | \,  d( \beta, z_0) \le \eps |z_0|) \to 1/2.  \end {equation}
One can for instance use a coupling argument, at the first time at which the previous diffusion $\theta$ hits $0$ (after which one couples $\theta$ with $-\theta$), and to notice that the probability that the probability that the diffusion spends a long time in $(0, \pi)$ is much smaller than the probability that it spends a long time in $(-\pi, \pi)$, regardless of the starting point of the diffusion.
This shows that the convergence in (\ref {thirdeq}) is in fact uniform with respect to $z_0$.
In other words, there exists a function $f (v)$ that decreases to $0$ as $v$ decreases to $ 0$, such that, for all $z_0$ and all $\eps < |z_0| / 2$, 
\begin {equation}
\label {firsteq}
\left|\PR ( z_0 \in H^+ (\beta) \,  | \, d(z_0, \beta) < \eps ) -  \PR ( z_0 \in H^- (\beta) \,  | \, d(z_0, \beta) < \eps ) \right| \le f(\eps/ |z_0|).
\end {equation}
From this, it follows in particular that (with the notations of Lemma \ref {minkovski2})
$E(v_{\eps, J}^+ (\beta)) \sim E( v_{\eps, J}^- (\beta))$ as $\eps \to 0$.

The second main ingredient in the proofs of \cite {LR} is a control on the second moments of 
$v_{\eps, J}$ and their variation with respect to $\eps$ (i.e. of the variations of a smoothed out version with respect to $\eps$).
One can summarize this type of result (that can be found in \cite {LR}) as follows: 
Define $T_\eps (z)$ as before as the hitting time of the disc of radius $\eps$ around $z$ by the chordal SLE$_\kappa$ $\beta$:
\begin {lemma}
\label {secondlemma}
For any compact subset $K$ of the upper half-plane, for each small $a >0$, there exists a positive $b=b(a)$ and a constant $c(K)$ such that for all $z, y$ in $K$ and all $\eps^{1-2a} < |z-y| $, 
$$ 
\PR \left( T_{\eps^{1-a}} (z) < T_\eps (y) < T_\eps (z) < \infty \right) 
   \le c(K) \, \eps^{b} \, \PR ( T_\eps (y) < T_\eps (z) < \infty ).
$$
\end {lemma}

We do not repeat the proof here, but it can be found in
Sections~2.3 and 4.2 of the preprint version of \cite{LR}.

This estimate means in particular, that if one conditions on $T_\eps (y) < T_\eps (z) <   \infty$, then with a large conditional probability, at time $T_\eps (y)$, the SLE 
did not get too close to $z$ yet. But by the previous estimate applied to the SLE paths after this time, the conditional probabilities that it gets then close to $z$ and passes to its right is very close to the conditional probability that it gets very close to $z$ and passes to its left. 

Let us be more specific: Let us define the event $E := \{ T_\eps (y) < \infty, \, T_\eps (y) <  T_{\eps^{1-a}} (z) \}$. 
Note that $E \subset \{ T_\eps (y) < \infty \}$ and that $E$ is measurable with respect to $\beta_{[0, T_\eps (y)]}$. 
Clearly,
\begin {align*}
\lefteqn {  
\left|  \PR \big( T_\eps (y) <  T_\eps (z) < \infty, \; z \in H^+ ( \beta) \big) 
     -  \PR \big( T_\eps (y) <  T_\eps (z) < \infty, \; z \in H^- ( \beta) \big) \right| }  \\
 &\leq
 \PR ( T_{\eps^{1-a}} (z) < T_\eps (y) < T_\eps (z) < \infty )   \\ 
 &  \qquad
   + \left| \EX \left[ \ind_{E}  
      \left( \;\PR\left[ T_\eps (z) < \infty, \; z \in H^+ ( \beta) \,\Big|\,  \beta_{[0, T_\eps (y)]} \right] 
           - \PR\left[ T_\eps (z) < \infty, \; z \in H^- ( \beta) \,\Big|\,  \beta_{[0, T_\eps (y)]} \right] \; 
      \right) \right] \right| .
 \end {align*}
 But on the event, $E$, conditionally on $\beta_{[0, T_\eps (y)]}$, if one applies the conformal Markov property at the time $T_\eps (y)$, simple distortion estimates yield that the disc of radius $\eps$ around $z$ gets mapped to a shape that is very close to a disc around the image of $z_0$, and that the radius of this disc is much smaller than the distance of this image to the real line. Furthermore, all these estimates are uniform enough, so that we can use (\ref {firsteq}) in order to conclude that in fact
 that the last line in the previously displayed equation is bounded by 
 $$   f_1 (\eps) C'(K)  \PR ( T_\eps (y) < T_\eps (z) < \infty ) $$ 
 and therefore 
 \begin {eqnarray*}
 &&  \left|  \PR \left( T_\eps (y) <  T_\eps (z) < \infty, \, z \in H^+ ( \beta) \right) 
          -  \PR \left( T_\eps (y) <  T_\eps (z) < \infty, \, z \in H^- ( \beta) \right) \right| \\
 &&\le    f_2 (\eps) \, C'(K) \, \PR \left( T_\eps (y) < T_\eps (z) < \infty \right)
 \end {eqnarray*}
for some functions $f_1$ and $f_2$ that tend to $0$ at $0$.

Using reversibility of SLE$_\kappa$ and applying the same reasoning to the curve $-1/\beta$, 
one can gets the similar bound  
\begin {eqnarray*} 
&& {
 \left|  \PR \left( T_\eps (z) <  T_\eps (y) < \infty, \, z \in H^+ ( \beta) \right) 
      -  \PR \left( T_\eps (z) <  T_\eps (y) < \infty, \, z \in H^- ( \beta) \right) \right| } \\
&& \le f_2( \eps) \, C''(K) \, \PR \left( T_\eps (z) <  T_\eps (y) < \infty \right) .
\end {eqnarray*}
Hence, adding up the two previous bounds, we get that
\begin {eqnarray*}
 && {
 \left|  \PR \left( T_\eps (y) < \infty, \, T_\eps (z) < \infty, \,{ z \in H^+ ( \beta) }\right) 
      -  \PR \left( T_\eps (y) < \infty, \, T_\eps (z) < \infty, \,{ z \in H^- ( \beta) }\right) \right| }
\\
&& \le f_2( \eps) \,C'''(K)\,
 \PR ( T_\eps (y) < \infty,   T_\eps (z) < \infty ) .
 \end {eqnarray*}
 Integrating this inequality with respect to $\de^2 y \, \de^2 z \, J(y) \, J(z)$, one gets
$$
\EX \left( | v_{\eps, J}^+ (\beta) - v_{\eps, J}^- (\beta) | \times  v_{\eps, J} ( \beta)  \right) 
  \le f_2(\eps) \, c(J) \, \EX \left( (v_{\eps, J} (\beta))^2 \right).$$
But clearly, the left-hand side is larger than
 $$ \EX \left( ( v_{\eps, J}^+ (\beta) - v_{\eps, J}^- (\beta))^2 \right)$$
 so that  Lemma \ref {minkovski2} follows via the Borel-Cantelli Lemma (just choose $\eps_n$ such that $\sum f_2 (\eps_n) < \infty$).
Here we used that $\EX \left( (v_{\eps, J} (\beta))^2 \right)$ is uniformly bounded in $\eps$,
see analogous statement in Theorem~3.1 of \cite{LR} (preprint version).

 \medbreak
We now very briefly indicate how to prove Lemma \ref {thealemma}. 
Let us come back to the description of chordal SLE via the SDE (\ref {theSDE}) 
$$
\de \theta_u = \sqrt {\kappa} \,\de B_u + \frac {4-\kappa}{2} \tan \frac {\theta_u}{2} \, \de u
$$
and describe briefly the outline of this rather standard argument.
\begin {enumerate}
 \item 
 Let $\sigma$ denote the first time at which the SLE curve generated by $\theta_u$ (in this $u$-parametrization) reaches the circle of 
 radius $a^{1/2}$ around the origin (if it does so this time is actually deterministic and equal to $\log (1/a) /2$, and otherwise we let $\sigma= \infty$). 
 A standard Harnack inequality type argument shows that, uniformly over the choices of 
 the starting point $\hat \theta $ and the end-point $\tilde \theta$  in $(- \pi, \pi)$,
 $$
 \PR \left( 
    \theta [0, \sigma] \subset (-\pi, 3\pi/4)  \hbox { or } 
    \theta [0, \sigma] \subset (-3\pi/4, \pi)  \,  \left| \, 
    \sigma < \infty, \, \theta_0=\hat \theta , \, \theta_\sigma=\tilde \theta \right. \right) \to 0
 $$
 as $a \to 0$. 
 \item
 Since the previous statement is uniform in $\tilde \theta$, it follows also that 
 $$
 \PR \left( 
    \theta [0, \sigma] \subset (-\pi, 3\pi/4)  \hbox { or } 
    \theta [0, \sigma] \subset (-3\pi/4, \pi)  \,  \left| \, 
    \sigma' < \infty, \, \theta_0=\hat \theta , \, \theta_{\sigma'}=\tilde \theta \right. \right) \to 0
 $$
 as $a \to 0$, where this time $\sigma'$ is the time at which the SLE reaches the circle of radius $2a$ around the origin, if it does so.
\item
Next, we notice that there exists $c>0$ so that for all $a$ small enough, and uniformly with respect to $\hat \theta$, 
$$ 
\PR \left( \theta_{\sigma'} \in (-\pi/2, \pi/2) \, \left| \, 
  \theta_0 = \hat \theta , \, \sigma' < \infty \right. \right) \ge c 
$$
and
$$ 
\PR \left( \gamma \hbox { disconnects } -a  \hbox { from } a  \, \left| \, 
  \sigma' < \infty, \, \theta_0=\hat \theta , \, \theta_{\sigma'} \in (-\pi/2, \pi/2) \right. \right) \ge c.
$$
If we put the pieces together, it follows that uniformly with respect to the starting point $ \theta_0$, we have for all small $a$, 
$$ 
c^2 \; \PR ( \sigma' < \infty ) \le \PR ( \gamma \hbox { disconnect } -a \hbox { from } a )
\le \PR ( \sigma' < \infty).$$
\item
We combine these estimates to get
\begin{align*}
 &\PR \left( 
    \theta [0, \sigma] \subset (-\pi, 3\pi/4)  \hbox { or } 
    \theta [0, \sigma] \subset (-3\pi/4, \pi)  \,  \left| \, 
    \gamma \hbox { disconnects } -a  \hbox { from } a ,  \, \theta_0=\hat \theta 
    \right. \right) \\
 &\leq  \frac{1}{c^2}
    \PR \left( 
    \theta [0, \sigma] \subset (-\pi, 3\pi/4)  \hbox { or } 
    \theta [0, \sigma] \subset (-3\pi/4, \pi)  \,  \left| \, 
    \sigma' < \infty ,  \, \theta_0=\hat \theta 
    \right. \right) 
    \to 0
\end{align*}
uniformly over the choices of the starting point $\hat \theta $.
If we let the starting point $\theta_0$ converge to $\pm \pi$, the lemma follows. 
\end {enumerate}

\bigbreak
\noindent
{\bf Acknowledgements.}
We acknowledge support of the Academy of Finland, of
 the Agence Nationale pour la Recherche under the grant  ANR-MAC2, of the Fondation Cino Del Duca and the Einstein Foundation, and the hospitality of Universit\'e Paris-Sud, of TU Berlin, the University of Cambridge and FIM Z\"urich, where part of this work has been elaborated in these past years.
We thank Greg Lawler for inspiring discussions, and Jason Miller and Sam Watson, David Wilson and Hao Wu for  their useful comments.

Department of Mathematics and Statistics

P.O. Box 68 

00014 University of Helsinki, Finland

\medbreak

Department of Mathematics

ETH Z\"urich, R\"amistr. 101

8092 Z\"urich, Switzerland

\medbreak

antti.h.kemppainen@helsinki.fi

wendelin.werner@math.ethz.ch

\end{document}